\documentclass[a4paper]{amsart}
 \pdfoutput=1
\usepackage[utf8]{inputenc}
\usepackage[T1]{fontenc}
\usepackage{lmodern}
\usepackage[final]{microtype}
\usepackage{amsmath}
\usepackage{amsthm}
\usepackage{amsfonts,amssymb}
\usepackage[english]{babel}
\usepackage{tikz}
\usepackage[backref=true,backend=bibtex]{biblatex}
\usepackage{tikz-cd}
\makeatletter
\newlength{\faktor@zaehlerhoehe}
\newlength{\faktor@nennerhoehe}
\DeclareRobustCommand*{\faktor}[3][]
{
   { \mathpalette{\faktor@impl@}{{#1}{#2}{#3}} }
}
\newcommand*{\faktor@impl@}[2]{\faktor@impl#1#2}
\newcommand*{\faktor@impl}[4]{
   \settoheight{\faktor@zaehlerhoehe}{\ensuremath{#1#2{#3}}}   \settoheight{\faktor@nennerhoehe}{\ensuremath{#1#2{#4}}}      \raisebox{0.5\faktor@zaehlerhoehe}{\ensuremath{#1#2{#3}}}      \mkern-5mu\diagup\mkern-4mu      \raisebox{-0.5\faktor@nennerhoehe}{\ensuremath{#1#2{#4}}}}
\makeatother
\usepackage{enumitem}
\usepackage{csquotes}
\usepackage[hidelinks, pdfstartview={FitV},bookmarksopen=true]{hyperref}

\makeatletter\AtBeginDocument{\hypersetup{pdftitle = {\@title}, pdfauthor = {\@author} }}\makeatother

\usetikzlibrary{matrix,arrows,chains,calc,fadings,decorations.pathreplacing}
\tikzset{commutative diagrams/arrow style=math font}

\title{Equivariant Differential Cohomology}

\author{Andreas Kübel}
\address{A.K., Max Planck Institute for Mathematics in the Sciences, Inselstraße 22, 04103 Leipzig, Germany}
\email{kuebel@mis.mpg.de}

\author{Andreas Thom}
\address{A.T., Institut f\"ur Geometrie, TU Dresden, 01062 Dresden, Germany}
\email{andreas.thom@tu-dresden.de}

\date{November 10, 2015}

\newcommand{\emphind}[1]{\emph{#1}}
\newcommand{\symindex}[2][a]{}

\theoremstyle{plain}\makeatletter
\g@addto@macro{\thm@space@setup}{\thm@headpunct{}}
\makeatother
\newtheorem{theorem}{Theorem}[section]
\newtheorem{lemma}[theorem]{Lemma}
\newtheorem{prop}[theorem]{Proposition}
\newtheorem{defn2}[theorem]{Definition + Proposition}

\theoremstyle{definition}
\newtheorem{example}[theorem]{Example}
\newtheorem{remark}[theorem]{Remark}
\newtheorem{defn}[theorem]{Definition}

\DeclareMathOperator{\im}{im}

\DeclareMathOperator{\id}{id}

\DeclareMathOperator{\tr}{tr}

\DeclareMathOperator{\Cone}{Cone}
\DeclareMathOperator{\map}{map}
\DeclareMathOperator{\End}{End}
\DeclareMathOperator{\Hom}{Hom}

\DeclareMathOperator{\Ad}{Ad}
\DeclareMathOperator{\pr}{pr}
\DeclareMathOperator{\sgn}{sgn}

\newcommand{\Cech}{\v{C}ech{}}
\newcommand{\from}{\colon}
\newcommand{\dual}{^\vee}
\newcommand{\defeq}{\mathrel{\mathop:}=}
  
\newcommand{\del}{\partial}
\newcommand{\Natural}{\mathbb{N}}
\newcommand{\Real}{\mathbb{R}}

\newcommand{\Integer}{\mathbb{Z}}
\newcommand{\Z}{\mathbb{Z}}

\newcommand{\Zp}{{C_p}}

\newcommand{\Complex}{\mathbb{C}}
\newcommand{\C}{\mathbb{C}}

\newcommand{\D}{\mathcal{D}}
\newcommand{\DG}{\mathcal{D}_{\text{Gomi}}}
\newcommand{\g}{\mathfrak{g}}
\newcommand{\frakd}{\mathfrak{d}}
\newcommand{\F}{\mathcal{F}}
\newcommand{\invers}{^{-1}}

\newcommand{\fatH}{\widehat{\mathbb{H}}}
\newcommand{\sing}{\mathrm{sing}}
\newcommand{\cell}{\mathrm{cell}}

\newcommand{\cl}{\mathrm{cl}}
\newcommand{\ddt}{\left.\frac{d}{dt}\right|_{t=0}}
\newenvironment{scaledcd}[1]{\begin{tikzpicture}[baseline= (a).base] \node[scale=#1] (a) at (0,0)\bgroup\begin{tikzcd}[column sep=small]}{\end{tikzcd}\egroup;\end{tikzpicture}\ignorespacesafterend}

\newcommand\pgfmathsinandcos[3]{  \pgfmathsetmacro#1{sin(#3)}  \pgfmathsetmacro#2{cos(#3)}}
\newcommand\LongitudePlane[3][current plane]{  \pgfmathsinandcos\sinEl\cosEl{#2}   \pgfmathsinandcos\sint\cost{#3}   \tikzset{#1/.estyle={cm={\cost,\sint*\sinEl,0,\cosEl,(0,0)}}}
}
\newcommand\LatitudePlane[3][current plane]{  \pgfmathsinandcos\sinEl\cosEl{#2}   \pgfmathsinandcos\sint\cost{#3}   \pgfmathsetmacro\yshift{\cosEl*\sint}
  \tikzset{#1/.estyle={cm={\cost,0,0,\cost*\sinEl,(0,\yshift)}}} }
\newcommand\DrawLongitudeCircle[2][1]{
  \LongitudePlane{\angEl}{#2}
  \tikzset{current plane/.prefix style={scale=#1}}
     \pgfmathsetmacro\angVis{atan(sin(#2)*cos(\angEl)/sin(\angEl))}   \draw[current plane] (\angVis:1) arc (\angVis:\angVis+180:1);
  \draw[current plane,dashed] (\angVis-180:1) arc (\angVis-180:\angVis:1);
}
\newcommand\DrawLatitudeCircle[2][1]{
  \LatitudePlane{\angEl}{#2}
  \tikzset{current plane/.prefix style={scale=#1}}
  \pgfmathsetmacro\sinVis{sin(#2)/cos(#2)*sin(\angEl)/cos(\angEl)}
    \pgfmathsetmacro\angVis{asin(min(1,max(\sinVis,-1)))}
  \draw[current plane] (\angVis:1) arc (\angVis:-\angVis-180:1);
  \draw[current plane,dashed] (180-\angVis:1) arc (180-\angVis:\angVis:1);
}

\tikzset{  >=latex,   inner sep=0pt,  outer sep=2pt,  mark coordinate/.style={inner sep=0pt,outer sep=0pt,minimum size=3pt,
    fill=black,circle}}
 \newcounter{mybibcounter}
 \AtEveryBibitem{\stepcounter{mybibcounter}}

\allowdisplaybreaks[1]
\begin{document}
\begin{abstract}
    The construction of characteristic classes via the curvature form of a connection is one motivation for the refinement of integral cohomology by de Rham cocycles -- known as differential cohomology. We will discuss the analog in the case of a group action on the manifold: The definition of equivariant characteristic forms in the Cartan model due to Nicole Berline and Mich{\`e}le Vergne motivates a refinement of equivariant integral cohomology by all Cartan cocycles. In view of this, we will also review previous definitions critically, in particular the one given in work of Kiyonori Gomi.
\end{abstract}
\maketitle

\tableofcontents

\newpage

\section{Introduction}
\enlargethispage{\baselineskip}
The interplay between geometry and topology is a widely occurring theme in modern mathematics, whose most elementary appearance is the formula of \emph{Hopf's Umlaufsatz}: Let $c\from [0,a]\to \Real^2$ be a closed smooth curve in the plane. Then the winding number of the curve is given by the integral over the curvature:
\[n_c=\frac1{2\pi}\int_0^a \kappa(t)\|c'(t)\| dt.\]
This result is surprising: The quantity on the left-hand side is an integer and purely topological -- vividly speaking this means: it does not depend on small alterations of the curve. Whereas, on the right-hand side, one integrates a real-valued function, which does depend on the geometry -- how long the curve is and how strongly it is curved.

A first generalization of the Umlaufsatz is known as the Gauss-Bonnet theorem, which states that for any compact surface $M$ of genus $g$ in $\Real^3$: 
\[2(g-1)=\frac1{2\pi}\int_M \kappa,\]
where now $\kappa$ denotes the Gaussian curvature of the surface. 

The generalizations of these statements by characteristic classes are based on \emph{de Rham cohomology}: The differential forms on a smooth manifold form a chain complex, which depends on the geometry of the space, but the cohomology of this chain complex is isomorphic to any real cohomology theory, e.g., to singular cohomology with real coefficients. This means that any real cohomology class -- a topological object -- can be represented by a closed differential form, a geometric object. 

In these terms, the left-hand side of the equations above will be generalized by the image of an \emph{integral} cohomology class in real cohomology; the curvature on the right-hand side will be replaced by a closed differential form (depending on the curvature) and the integral will be expressed by taking the cohomology class of this form.
\\[2ex]
In general, characteristic classes associate cohomology classes to (isomorphism classes of) vector bundles. For smooth bundles, there are two well-known procedures to construct them, one which applies the geometric structure and one which uses topology only:

The Chern-Weil-Homomorphism starts with a connection on the bundle and evaluates an invariant symmetric polynomial on the associated curvature form, which leads to a closed differential form, the characteristic form. As the difference of the characteristic forms of two connections is an exact form -- the exterior derivative of the transgression form -- one gets a class in de Rham cohomology which is independent of the chosen connection and called the characteristic class of the bundle.

On the other hand one may also obtain these classes by pulling back universal characteristic classes via the classifying map of the bundle. 

Both construction have their own strengths: The characteristic form contains geometric data, while the class is purely topological. The class itself actually is not an element in real, but in integral cohomology, where algebraic torsion may deliver finer information, which cannot be reflected by the characteristic form, as there is no algebraic torsion over the field of real or complex numbers.

To use both, the geometric information of the characteristic form and its transgression and the algebraic torsion information from integral cohomology in one object, one defines differential cohomology and differentially refined characteristic classes. This was done first by Jeff Cheeger and James Simons in \cite{CS}. The differential cohomology theory extends integral cohomology by closed differential forms. A notable result is that -- while the classical first Chern class classifies complex line bundles up to isomorphism -- the first differential Chern class classifies complex line bundles with connection up to isomorphism.

From this starting point there are various ideas of differential refinements of cohomology theories: Besides the differential characters of Cheeger and Simons, there is an isomorphic model by smooth Deligne cohomology (see \cite{Brylinski,Bunke}). On the other hand there are various models for differential K-theory (see \cite{BSsurvey} for a survey, which includes a discussion of the literature). A general framework for these differential refinements is given in  \cite{BunkeSchick} and \cite{BNV}.

We want to go back to the starting point and generalize the idea of the differential refinement to an equivariant setting, i.e., we have a Lie group $G$ acting on a smooth manifold $M$ and ask for a theory which enables differential refinements of equivariant characteristic classes of $G$-equivariant vector bundles over $M$. 

To do so, we need a differential form model for equivariant cohomology, which is capable to receive a homomorphism from integral cohomology. Moreover, there should be two constructions of real/complex equivariant characteristic classes, one via equivariant characteristic forms and one via integral equivariant characteristic classes, which should coincide under the homomorphism between the cohomology theories.

The construction of the differential refinement, which we will give, is an equivariant version of smooth Deligne cohomology, but to stress that it fits into the picture of differential refinements, we will use the term equivariant differential cohomology, even if we will not discuss equivariant differential refinements in general.

\subsection{Equivariant cohomology and simplicial manifolds}
Defining equivariant cohomology $H_G^*(M)$ is a simple business using two expected properties of this functor: homotopy invariance and that, for free actions, the equivariant cohomology should coincide with the cohomology of the quotient. Namely, let $EG$ be a contractible space with a free $G$-action, then the diagonal action of $G$ on $EG\times M$ is free and the map $EG\times M\to \{*\}\times M$ is a homotopy equivalence. Hence, we have described the well-known Borel construction, which is in formulas \[H_G^*(M)= H_G^*(\{*\}\times M)=H_G^*(EG\times M)=H^*(EG\times_G M),\]
for any cohomology theory and any coefficient group, e.g. singular cohomology with values in $\Z,\Real$ or $\C$. Here $EG\times_G M$ is the quotient of the diagonal $G$-action on $EG\times M$.

As short and easy this construction is, it creates a task for us: $EG$ is even in simple cases not a finite-dimensional manifold, hence we have no de Rham cohomology. But $EG$ is something similar to a manifold: Namely there is a simplicial manifold (\cite{Deligne,Dupont}), i.e., a simplicial set such that the set of $p$-simplices forms a smooth manifold for each $p$ and all face and degeneracy maps are smooth, and the  geometric realization of this simplicial manifold is $EG\times_G M$. This will be introduced in Section \ref{sec:simpman} and we will explain how one defines (simplicial) differential forms on a simplicial manifold. They lead to a complex, which is bi-graded: by the form degree and the simplicial degree. The cohomology of this double complex calculates equivariant complex cohomology. In fact, simplicial differential forms also form a (graded) simplicial sheaf $\Omega_{\C}^{\bullet,*}$. 

Using the language of simplicial sheaf cohomology, the de Rham homomorphism is induced by the inclusion of the locally constant simplicial sheaf $\underline\Z\to\Omega_{\C}^{\bullet,*}$, as locally constant functions.

In Section \ref{sec:Cartan}, we will introduce the reader to a more famous model of equivariant cohomology using differential forms, known as the Cartan model. This is given by the so-called equivariant differential forms, i.e., equivariant polynomial maps $\mathfrak g\to \Omega^*(M)$, where the differential $d_C$ on $(\C[\mathfrak g]\otimes \Omega^*(M))^G$ is given by
\[(d_C \omega)(X)=d(\omega(X))+\iota_{X^\sharp}(\omega(X)),\]
i.e. the sum of the exterior differential and the contraction with the fundamental vector field of $X$, and hence increases the grading given through
\[\text{twice the polynomial degree } + \text{ the differential form degree}\]
by one.

The Cartan model has the advantage that its cochain complex is substantially slimmer than the double complex $\Omega^{\bullet,*}$ defined above, but it is not directly capable to receive a homomorphism from integral cohomology. Therefore we apply ideas of \cite{Getzler} to compare the different models of equivariant cohomology. This comparison will enable our construction of a differential refinement of equivariant integral cohomology.

\subsection{Equivariant characteristic classes and forms}
Let $G$ act on the vector bundle $E\to M$, i.e., we have an action on the total space and the base space, such that the projection is equivariant. Via the Borel construction, one can define equivariant characteristic classes easily: Take the usual characteristic classes of $EG\times_G E\to EG\times_G M$! 

There is also a characteristic form construction (see \cite{BerlineVergne}) which does not only depend on the curvature, but also uses the moment map $\mu^\nabla$ of the connection $\nabla$. This is a map from the Lie algebra of the acting group to the endomorphisms of the vector bundle (see Definition \ref{def:momentmap} for details). In this way, one obtains an equivariant characteristic form, which is a closed equivariant differential form, i.e., an element in the Cartan complex.

Both paths lead to tho same class in equivariant complex Borel cohomology. We discuss this in \cite{Ich1}, since for this compatibility, although generally assumed to hold, there exist only a proof for special cases (compare \cite{BottTu01}) in the literature. 

\subsection{Equivariant differential cohomology}
After we have achieved this understanding of equivariant characteristic forms, we can review previous definitions critically to obtain a more satisfactory one.

There is a definition of equivariant smooth Deligne cohomology $\hat H_G^*(M,\Z)$ in \cite{Gomi}, and Kiyonori Gomi shows there that $\hat H_G^2(M,\Z)$ classifies $G$-equivariant line bundles with connection. We will show that his definition fits, for actions of compact groups, into a differential cohomology hexagon (\autoref{thm:GomiHexagon1}) and thus can be interpreted as a model for equivariant differential cohomology. But this definition neglects the secondary information of the moment map and is, thus, only satisfactory in the case of finite groups, where there is no moment and in low degrees, where the moment map does not play a role. There are also other, less elaborated, definitions (see Remark \ref{rm:OtherDefs}), which are all unsatisfactory from our insight to characteristic forms. 

Therefore, in Section \ref{sec:newdef}, we define (full) equivariant differential cohomology $\fatH_G^*(M,\Z)$ (using a mapping cone construction similar to the non-equivariant case in \cite{Bunke}) and show (see Theorem \ref{thm:FullHexagon}) that for any compact Lie group $G$, one has the commutative diagram
\[
\begin{scaledcd}{.9}
\phantom{.} &\hspace{-2em}\faktor{\Omega^{n-1}_G(M)}{(d+\iota)\Omega^{n-2}_G(M)} \arrow{dr}{a} \arrow{rr}{d+\iota}& & \Omega^{n}_G(M)_\cl\arrow{dr}{}&\\
H_G^{n-1}(M,\C)\arrow{ru}{}\arrow{dr}{}&&\hspace{-1em}\fatH_G^n(M,\Z) \arrow{ru}{R}\arrow[two heads]{dr}{I}&&H_G^n(M,\C)\\
&H_G^{n-1}(M,\C/\Z) \arrow[hook]{ur}{} \arrow{rr}{-\beta}&&H_G^n(M,\Z) \arrow{ur}{}&
\end{scaledcd}
\]
where the line along the top, the one along the bottom and the diagonals are exact.

In the case of the trivial group one obtains the classical differential cohomology. In degree up to two, our definition coincides with the one of Gomi. In higher degrees one has additional geometric data, e.g., in the case of the conjugation action of $S^3=SU(2)$ on itself, as discussed in Section \ref{sec:S3}, one has $\hat H_{S^3}^4(S^3,\Z)= H_{S^3}^3(S^3,\C/\Z)\oplus H_{S^3}^4(S^3,\Z)=\C/\Z\oplus \Z$, while we have a short exact sequence \[0\to \faktor{\Omega^1(S^3)^{S^1}}{dC^\infty(S^3)^{S^1}}\to \fatH_{S^3}^4(S^3,\Z)\to\hat H_{S^3}^4(S^3,\Z)\to 0,\] hence we have additional transgression data.

From the hexagon, one concludes that equivariant differential cohomology is the right group to define equivariant differential characteristic classes in, since they can refine both, the equivariant integral characteristic class and the equivariant characteristic form. The details of this constructions are worked out in Section \ref{sec:eqdiffchar}.

\section{Models for equivariant cohomology}
Let $M$ be a smooth manifold acted on from the left by a Lie group $G$. To define equivariant cohomology one uses two properties which one expects from such a theory: it should be homotopy invariant and for free actions, the equivariant cohomology should be the cohomology of the quotient. Recall that the total space of the classifying bundle $EG$ is a contractible topological space with free $G$-action. Hence $EG\times M$ has the homotopy type of $M$ and the diagonal action is free. Hence one defines 
\[H^*_G(M)\defeq H^*(EG\times_G M),\] 
where $EG\times_G M$ \symindex[e]{EG\times_G M} is the quotient of $EG\times M$ by the diagonal action.
We are interested in differential form models for equivariant cohomology, but in general $EG$ is not a finite-dimensional manifold, hence we cannot use the usual de Rham cohomology. But there is a model for $EG$, which consist of finite dimensional manifold:

\subsection{Simplicial manifolds and differential forms}\label{sec:simpman}
The model of $EG\times_G M$ we are going to use is a given by a simplicial manifold.
\begin{defn}[see, e.g., {\cite[p.89]{Dupont}}]
    A \emph{simplicial manifold}\index{simplicial manifold} is contra-variant functor from the simplex category $\Delta$ to the category of smooth manifolds.
\end{defn}
Explicitly this is an $\Natural$-indexed family of manifolds with smooth \emph{face}\index{face maps} and \emphind{degeneracy maps} satisfying the simplicial relations, i.e.
\begin{align*}
\del_i\circ\del_j&=\del_{j-1}\circ\del_i, \text{ if } i<j\\
\sigma_i\circ \sigma_j&=\sigma_{j+1}\circ \sigma_i, \text{ if } i\leq j\\
\del_i\circ \sigma_j &= \begin{cases}
                    \sigma_{j-1}\circ\del_i , &\text{ if } i<j\\
		    \id , &\text{ if } i=j,j+1\\
		    \sigma_{j}\circ\del_{i-1} , &\text{ if } i>j+1
                   \end{cases}
\end{align*}\symindex[d]{\del_i}\symindex[s]{\sigma_i}
\begin{example}\label{ex:actionmanifold}
Our most important example of a simplicial manifold is the following (compare \cite[p.316]{Gomi},\cite[section 3.2]{Getzler}): \symindex[g]{G^\bullet M}
\[G^\bullet\times M = \{G^p\times M\}_{p\geq 0},\]
where $G^p$ stands for the $p$-fold Cartesian product of $G$. The face maps $G^p\times M\to G^{p-1}\times M$ are given as
\begin{align*}
\del_0(g_1, \dots , g_p, x) &= (g_2, \dots, g_p, x)\\
\del_i(g_1, \dots , g_p, x) &= (g_1, \dots , g_{i-1}, g_{i}g_{i+1},\dots , g_p, x)\text{ for } 1 \leq i \leq p - 1\\
\del_p(g_1, \dots , g_p, x) &= (g_1, \dots , g_{p-1}, g_p x)
\end{align*}
and the degeneracy maps for $i=0,\dots,p$ by
\begin{align*}
 \sigma_i: G^p\times M&\to G^{p+1}\times M\\
(g_1,\dots,g_p,x)&\mapsto (g_1,\dots,g_i,e,g_{i+1},\dots,g_p,x).
\end{align*}
These maps satisfy the simplicial relations.
\label{rm:action}
In particular for $p=1$ the map $\del_1$ equals the group action, while $\del_0$ is the projection onto the second factor, i.e. onto $M$. 
\end{example}

\begin{defn}[see, e.g., {\cite[p.75]{Dupont}}]
    The (fat) \emphind{geometric realization} of a simplicial manifold $M_\bullet$, is the topological space
 \[\|M_\bullet\|=\bigcup_{p\in\Natural} \Delta^p\times M_p/\sim\]
with the identifications
\[(\del^it,x)\sim(t,\del_i x)\text{ for any }x\in M_p,\, t\in\Delta^{p-1}, i=0,\dots,n \text{ and } p=1,2,\dots .\]
\end{defn}
\begin{example}
The geometric realization of the simplicial manifold $G^\bullet\times M$ is a model of $EG\times_G M$ and in particular if $M$ is single point the geometric realization of $G^\bullet\times pt$ is a model of the classifying space $BG$ (compare \cite[pp.75]{Dupont}).
\end{example}

Before giving a differential form model for equivariant cohomology, we will explain sheaves and sheaf cohomology for simplicial manifolds, as this is the technical basis for all further constructions and definitions.
\begin{defn}[see {\cite[(5.1.6)]{Deligne}}]
 A \emph{simplicial sheaf} on the simplicial manifold $M_\bullet$ is a collection of sheaves $\mathcal{F}^\bullet=\{\mathcal{F}^p\}_{p\in\Natural}$, where, for each $p$, $\mathcal{F}^p$ is a sheaf on $M_p$ and there are morphisms $\tilde\del_i:\del_i^{-1} \mathcal{F}^p\to \mathcal{F}^{p+1}$ and $\tilde{\sigma}_i:\sigma_i^{-1} \mathcal{F}^{p+1}\to \mathcal{F}^{p}$ satisfying the simplicial relations as stated above.
\end{defn}

The \emph{simplicial sheaf cohomology} is defined as the right derived functor of the global section functor \cite[def. 5.2.2.]{Deligne}, where global sections of a simplicial sheaf, are the equalizer
\[\ker \left(\tilde\del_0-\tilde\del_1\from \mathcal{F}^0(M_0)\to\mathcal F^1(M_1)\right).\]
This definition opens the question: Are there enough injectives? As Pierre Deligne is quite short on this and there are mistakes in the literature (see Remark \ref{rem:mistake}), we should give an answer. 
\begin{lemma}The category of simplicial sheaves has enough injectives.\end{lemma}
\begin{proof}
Let $\mathcal{F}^\bullet$ be a simplicial sheaf. Let $P_p$ be the functor from simplicial sheaves to sheaves, which sends a sheaf to its $p$-th level, i.e., $\mathcal F^\bullet$ is sent to  the sheaf $\mathcal{F}^p$ on $M_p$. Pick for any $\mathcal{F}^p$ an injective sheaf $I^p$ on $G^p\times M$, in which $\mathcal F^p$ embeds (for existence see e.g. \cite[section III.2]{Hartshorne}).

Now we construct a right adjoint of $P_p$ (analogous to \cite[p.409]{simphomtheory}): Let $B$ be a sheaf on $G^p\times M$. Define a simplicial sheaf on $G^\bullet\times M$ as
\[(S_pB)_q=\prod_{h\in \Delta(q,p)} h^{-1}B\]
By the adjointness of the functors, injectivity of $B$ implies injectivity of $S_pB$. Moreover the equality  
\[\Hom\left(\!\mathcal F^\bullet,\prod_p S_pI^p\!\right)=\prod_p \Hom(\mathcal F^\bullet,S_pI^p)=\prod_p \Hom(P_p\mathcal F^\bullet,I^p)=\prod_p \Hom(\mathcal F^p,I^p)\]
shows that the simplicial sheaf $\mathcal F^\bullet$ embeds into $\prod_p S_pI^p$ because for each $\mathcal F^p$ there is an injection into $I^p$.
\end{proof}
Now let 
\[0\to \mathcal F^\bullet\to I^{\bullet,0}\overset{\delta}{\to} I^{\bullet,1}\overset{\delta}{\to} \dots\]
be an injective resolution. Omitting the first columns and taking global sections yields to a double complex 
\[\left(I^{p,q}(M_p),\sum_{i=0}^p (-1)^i\tilde\del_i+(-1)^p \delta\right),\] whose cohomology is defined to be the cohomology
\[H^*(M_\bullet, \mathcal F^\bullet)=H^*\left(I^{p,q}(M_p),\sum_{i=0}^p (-1)^i\tilde\del_i+(-1)^p \delta\right)\] of the simplicial sheaf $\mathcal F^\bullet$ on the simplicial manifold $M_\bullet$.

The definition does not depend on the injective resolution chosen. In the non-simplicial case, this is a well-known fact: the identity on the space and the sheaf induces a morphism between two chosen injective resolutions, which is an isomorphism in cohomology. In the simplicial case, we need an additional argument: As before we obtain a morphism of the double complexes of global sections from the identity on the space. When taking cohomology in every horizontal line $(I^{p,*}(M_p),(-1)^p \delta)$, this morphism will induce an isomorphism between the bi-graded complexes. Hence we can apply the following lemma, to see, that we have an isomorphism in cohomology.
\begin{lemma}[{see e.g. \cite[Lemma 1.19]{Dupont}}]\label{lemma:specseq}
 Suppose $f\from (C_1^{*,*},d_1'+d_1'')\to (C_2^{*,*},d_2'+d_2'')$ is homomorphism of double complexes and the induced homomorphism \[(H^q(C_1^{p,*},d_1''),d_1')\to (H^q(C_2^{p,*},d_2''),d_2')\] is an isomorphism, then $f$ induces an isomorphism in the total cohomology of double complexes.
\end{lemma}
\begin{remark}\label{rem:mistake}
One could have the idea (e.g. \cite[p.3]{Brylinski00},\cite[Section 3.2]{Gomi}) that an injective resolution on any simplicial level would be sufficient as the maps $\tilde \del_i$ lift by the injectivity of the sheaf. But as this lift is not unique, it is unclear that the simplicial relations hold and thus there is no general reason why $\del=\sum_i (-1)^i\tilde \del_i$ is a boundary operator. In fact one can construct the following counterexample: Take the trivial group, acting on a point, then all $\tilde\del_i\from \Z\to\Z$ are the identity. A injective resolution of the abelian group $\Z$ is given by $\Z\to\C\to \C/\Z$. Beside $\id\from \C\to\C$, the complex conjugation is also a lift of $\id_\Z$. Making appropriate choices, for the lifts $\tilde\del_i$ one finds an example where $\del\circ\del\neq 0$.
\end{remark}

In practice, one usually uses acyclic resolutions, instead of injective ones, to calculate cohomology. This works in the simplicial case, too. Let 
\[0\to \mathcal F^\bullet\to \mathcal A^{\bullet,0}\overset{\delta}{\to} \mathcal A^{\bullet,1}\overset{\delta}{\to} \dots\]
be an acyclic resolution, i.e., each $\mathcal A^{\bullet,k}$ is a simplicial sheaf and all but the zeroth cohomology of each sheaf $\mathcal A^{p,q}$ vanish. Let $I^{\bullet,*}$ be a simplicial injective resolution. The identity map on the simplicial manifold and the sheaf $\mathcal F^\bullet$ induce a homomorphism of the double complex of global sections (by injectivity of $I$), which induces an isomorphism of the bi-complexes, $(H^q(\mathcal A^{p,*},\delta),\del)\to(H^q(I^{p,*},\delta),\del)$, as acyclic resolutions calculate cohomology. Thus the last lemma implies the isomorphism in the cohomology of the double complexes.

In the examples, which we study later, the simplicial sheaf will actually not just be a sheaf of abelian groups, but a  cochain complexes of simplicial sheaves of abelian groups. A resolution for a chain complex goes by the name \emphind{Cartan-Eilenberg resolution} and exists for cochain complex in any abelian category with enough injectives (compare \cite[Section 5.7]{Weibel}). In our context, the resolution of a cochain complex of simplicial sheaves is a triple instead of a double complex. Nevertheless, one can form a total complex of the global sections of the triple complex and the cohomology of the cochain complex of simplicial sheaves is defined as the cohomology of this total complex. 

We will now discuss some explicit models for simplicial sheaf cohomology. 

\subsubsection{Simplicial de Rham cohomology}This exposition is based on \cite[Section 6]{Dupont}. Let $M_\bullet=\{M_p\}$ be a simplicial manifold. For any $p$, differential forms on $M_p$ form a the cochain complex of sheaves $(\Omega^*_{M_p},d)$. The face and degeneracy maps of $M_\bullet$ induce, via pullback, face and degeneracy maps between the differential forms on $M_p$ and $M_{p\pm 1}$. Thus, one obtains the simplicial sheaf $\Omega^{\bullet,*}$ of differential forms on $M_\bullet$. 

On the global sections of this sheaf 
\[\Omega^{p,q}(M)=\Omega^q(M_p),\]
there is a horizontal differential $d:\Omega^{p,q}(M_\bullet)\to \Omega^{p,q+1}(M_\bullet)$, given by the exterior differential and vertical differential 
\[\del: \Omega^{p,q}(M_\bullet)\to \Omega^{p+1,q}(M_\bullet),\]
given by the alternating sum of pullbacks along the face maps
\begin{equation}\label{del}
\del(\omega)=\sum_{i=0}^{p+1} (-1)^i\del_i^*\omega.
\end{equation}
\begin{prop}  $(\Omega^{p,q}(M_\bullet),d+(-1)^q\del)_{p,q}$ forms a double complex.
\end{prop}
\begin{proof}
 $(d+(-1)^q\del)^2=0$, because $\del^2=0$ by the simplicial relations, $d\del=\del d$ as $d$ is functorial, $d^2=0$ by the well-known property of the exterior derivative.
\end{proof}

Moreover, since the differential forms form a sheaf of $C^\infty$-module, they form a fine and hence acyclic sheaf.

In particular, for the simplicial manifold $G^\bullet\times M$, we have the double complex $\Omega^q(G^p\times M)$, what is a first de Rham type model for equivariant cohomology by the following Proposition.

\begin{prop}[Prop. 6.1 of {\cite{Dupont}}]
    Let $M_\bullet$ be a simplicial manifold. There is a natural isomorphism \[H^*(\Omega^{\bullet,*}(M_\bullet),d+(-1)^*\del)\cong H^*(\|M_\bullet\|,\C).\]
\end{prop}

\subsubsection{Simplicial \texorpdfstring{Č}{C}ech cohomology}\label{sec:simpcechcoho}
\begin{defn}[see {\cite{Brylinski00,Gomi}}]\label{def:simpcover}
 A \emph{simplicial cover} \index{cover!simplicial} for the simplicial manifold $M_\bullet$ is a family $\mathcal U^\bullet=\{\mathcal U^{(p)}\}$\symindex[u]{\mathcal U^\bullet,\mathcal U^{(p)}} of open covers such that
\begin{enumerate}
 \item $\mathcal U^{(p)}=\{U^{(p)}_\alpha|\alpha\in A^{(p)}\}$ is an open cover of $M_p$, for each $p$, and
 \item the family of index sets forms a simplicial set $A^\bullet=\{A^{(p)}\}$ satisfying
 \item $\del_i(U^{(p)}_{\alpha})\subset U^{(p-1)}_{\del_i\alpha}$ and $\sigma_i(U^{(p)}_{\alpha})\subset U^{(p+1)}_{\sigma_i\alpha}$ for every $\alpha\in A^{(p)}.$
\end{enumerate}
\end{defn}
\begin{defn}[see {\cite{Brylinski00,Gomi}}]
Given a simplicial cover $\mathcal U^\bullet$, one forms the Čech chain groups $\check C^{\bullet,*}(\mathcal U^\bullet,\mathcal F^\bullet)$ by 
\[\check C^{p,q}(\mathcal U^\bullet,\mathcal F^\bullet)=\prod_{\alpha^{(p)}_0,\dots,\alpha^{(p)}_q\in A^{(p)}} F^p\left(U^{(p)}_{\alpha^{(p)}_0}\cap\dots\cap U^{(p)}_{\alpha^{(p)}_q}\right),\]\symindex[c]{\check C^{\bullet,*}}\index{simplicial \Cech{} chain complex}with the usual Čech boundary operator $\delta\from \check C^{p,q}\to \check C^{p,q+1}$ and the simplicial boundary map $\del\from\check C^{p,q}\to \check C^{p+1,q}$ defined as alternating sum as above.
\end{defn}
Observe, that the third condition of the simplicial cover ensures that $\del$ maps between the Čech groups. The \emphind{simplicial \Cech{} cohomology}, denoted by 
\[\check H^*(\mathcal U^\bullet, \mathcal F^\bullet),\]
is the cohomology of the double complex $(\check C^{p,q},\del,(-1)^p \delta)$. As in the non-simplicial case (see \cite[section III.4]{Hartshorne}), any simplicial open cover induces a canonical homomorphism
\[\check H^*(\mathcal U^\bullet, \mathcal F^\bullet)\to H^*(M_\bullet,\mathcal F^\bullet).\] Moreover, given a refinement $\mathcal V^\bullet$ of the simplicial open cover $\mathcal U^\bullet$, then the natural diagram 
\[\begin{tikzcd}\check H^*(\mathcal U^\bullet, \mathcal F^\bullet)\arrow{d}\arrow{r}& H^*(M_\bullet,\mathcal F^\bullet)\\
 \check H^*(\mathcal V^\bullet, \mathcal F^\bullet)\arrow{ur}&     
  \end{tikzcd}\]commutes. 
Thus one can form the limit over all refinements of simplicial open covers and obtains an isomorphism
\[\lim_{\mathcal U^\bullet} \check H^*(\mathcal U^\bullet, \mathcal F^\bullet)\to H^*(M_\bullet,\mathcal F^\bullet).\]
For more details see \cite{Brylinski00,Gomi}.

\subsubsection{Simplicial singular cohomology}\label{sec:simpsingcoho}
Let $A$ be an abelian group. Later, the most interesting cases for us will be $A\in\{\Z,\Real,\C,\C/\Z,\Real/\Z\}$. Then there is the locally constant sheaf $\underline A^\delta$, consisting of continuous maps to $A$ furnished with the discrete topology, in any simplicial degree. The maps $\tilde\del_i$ and $\tilde\sigma_i$ are given by pullback along $\del_i$ respectively $\sigma_i$. One can calculate $H^*(M_\bullet,\underline A)$ via singular cohomology.

\begin{defn}[see {\cite[81]{Dupont}}]\symindex[c]{C_\text{sing}^{\bullet,\bullet}}
    The \emph{simplicial singular cochain complex} \index{simplicial singular cochain complex} \[(C^{\bullet,\bullet}_\sing(M_\bullet,A),\del,\del_\sing)\] is the double complex consisting of groups
	\[C^{p,q}_\sing=C^q_\sing(M_p)=\map(C^\infty(\Delta^q,M_p),A)\]
of smooth singular cochains on each $M_p$ with group structure induced from $A$, vertical boundary map induced from the simplicial manifold and horizontal boundary map given by the singular boundary operator. 
\end{defn}
To obtain a double complex one has to use the boundary map $\del+(-1)^p\del_\sing$. A simplicial map $f_\bullet\from M_\bullet\to M'_\bullet$ induces a map of double complexes $f_\bullet^*\from C^{\bullet,\bullet}_\sing(M'_\bullet,A)\to C^{\bullet,\bullet}_\sing(M_\bullet,A).$
\begin{theorem}[Theorem 5.15. of \cite{Dupont}]\label{thm:realiso}
 There are functorial isomorphisms
\[H^*(\|M\|,A)=H^*_\sing(M_\bullet,A)\defeq H^*\left(C^{\bullet,\bullet}_\sing(M_\bullet,A),\del+(-1)^p\del_\sing\right)\] \symindex[h]{H^*_\text{sing}}\end{theorem}

To compare singular cohomology with general sheaf cohomology, one can use arguments of \cite[pp. 191-200]{Warner}. Sheafify the singular cochains $C^q_\sing(M_p)$: Let $\mathcal S^q(M_p,A)$ \symindex[s]{\mathcal S^*} be the sheaf associated to the presheaf \[M\subset U\mapsto\map(C^\infty(\Delta^q,U),A).\]
Then one has an acyclic resolution 
\[0\to \underline A_\bullet\to \mathcal S^0(M_\bullet,A)\to \mathcal S^1(M_\bullet,A)\to \dots\]
and hence
\[H^*(M_\bullet,A)=H^*(M_\bullet,\mathcal S^*(M_\bullet,A)).\]
On the other hand, the global sections of $\mathcal S^q(M_p,A)$ are exactly $C^q_\sing(M_p)$.

Thus we have shown the following theorem.
\begin{theorem}
	\[H^*(\|M\|,A)=H^*_{\text{sing}}(M_\bullet,A)=H^*(M_\bullet, \mathcal S^*) =H^*(M_\bullet,\underline A).\]
In particular, for $M_\bullet=G^\bullet\times M$, we obtain:
    \[H_G^*(M,A)=H^*_{\text{sing}}(G^\bullet\times M,A)=H^*(G^\bullet\times M,\underline A).\]
\end{theorem}

\subsubsection{Simplicial cellular cohomology}\index{simplicial cellular chain complex}\label{sec:simpcellcoho}
The most handy cohomology theory for calculation is cellular cohomology. Recall (compare \cite[12]{Waldhausen}) that a \emphind{CW complex} is a topological space $X$ with a collection of subspaces, called cellular decomposition, 
\[ X_0\subset X_1\subset X_2 \subset \dots \subset X, \]
such that $X_0$ is discrete, $X_p$ is obtained from $X_{p-1}$ by attaching $p$-cells, $X=\bigcup_i X_i$, and $U\subset X$ is closed, if and only if $U\cap X_p$ is closed in $X_p$ for any $p\in \Natural$. A map $f\from X\to Y$ between cellular complexes is called cellular, if $f(X_p)\subset Y_p$. The cellular chain complex (see \cite[118-122]{Waldhausen}) is given by $C^n(X)=H^n_\sing(X_n,X_{n-1};A)$ and $d_\cell^n$ is the composition 
\[H^n(X_n,X_{n-1})\to H^n(X_n,\emptyset)\to H^{n+1}(X_{n+1},X_n)\]
of the map induced from the inclusion $(X_n,\emptyset)\subset (X_n,X_{n-1})$ and the connecting morphism of $(X_n,\emptyset)\subset (X_{n+1},\emptyset)\subset(X_{n+1},X_n)$.

By a cellular decomposition of the simplicial manifold $G^\bullet\times M$, we understand a collection of topological spaces $(X_{p,q})_{p,q\in\Natural}$, such that $X_{p,*}$ is a cellular decomposition of $G^p\times M$ and all face and degeneracy maps are cellular. Thus we receive a double complex, the \emphind{simplicial cellular chain complex}\symindex[c]{C^{\bullet,*}_\cell} $(C_\cell^q(G^p\times M),d_\cell+(-1)^q\del)$. We define \symindex[h]{H^*_\cell}$H_\cell^*(G^\bullet\times M,A)$ to be the cohomology of this double complex.

One has the following small proposition, for which I did not find a reference in the literature.
\begin{prop}
 There is an isomorphism
\[H^*_\cell(G^\bullet\times M,A)=H^*_\sing(G^\bullet\times M,A).\]
\end{prop}

\begin{proof} Given a map between the singular and cellular chains, Lemma \ref{lemma:specseq} would imply the result. Hence we are done, if we find such a map for normal, i.e., non-simplicial spaces, in a functorial manner. There is no map between singular and cellular chains in general, but one can construct a complex of so called simplicial singular chains (see \cite[Section V.8]{Dold}), and functorial quasi-isomorphisms to both, singular and cellular chains.
\end{proof}

\subsection{The Cartan model}\label{sec:Cartan}
A well-known de Rham-like model for equivariant cohomology goes back to Henri Cartan (\cite{Cartan}). Our Exposition follows \cite{Libine}. 
Let $G$ be a compact Lie group acting smoothly on the smooth manifold $M$ and denote the Lie algebra of $G$ by $\g=T_eG$. Let $S^*(\g\dual)$ be the symmetric tensor algebra of the (complex) dual of the Lie algebra $\g\dual$. The group $G$ acts on this algebra by the coadjoint action and on $\Omega^*(M)$ by pulling back forms along the map $m\mapsto gm$. Hence we have a $G$-action on $S^*(\g\dual)\otimes\Omega^*(M)$. The invariant part of this algebra $(S^*(\g\dual)\otimes\Omega^*(M))^G$ is exactly what one calls the \emphind{Cartan complex} and is denoted by $\Omega_G^*(M)$\symindex[o]{\Omega_G^*(M)}. In other words: The Cartan complex consists of $G$-equivariant polynomial maps $\omega\from \g\to \Omega^*(M)$. Let $\omega_1,\omega_2\in \Omega_G^*(M)$, then there is a wedge product 
\[(\omega_1\wedge \omega_2)(X)=\omega_1(X)\wedge \omega_2(X).\]
On this algebra one defines a differential as
\[d_C \omega (X)=d(\omega(X))+\iota(X^\sharp)\omega(X),\]
for $\omega\in \Omega_G^*(M))$ and $X\in \g$,
i.e., one takes the differential on the manifold and adds the contraction with the fundamental vector field. To make this differential raise the degree by one, the grading on $\Omega_G^*(M)$ is given by
\[\text{twice the polynomial degree} + \text{the differential form degree}.\]
\begin{lemma}
    $(\Omega^*_G,d_C)$ is a cochain complex.
\end{lemma}
\begin{proof}
First, observe that $d_C$ increases the total degree by one, since $d$ increases the differential form degree, and the contraction $\iota$, while decreasing the form degree by one, increases the polynomial degree by one. Next, one has to check, that the differential really maps invariant forms to invariant forms and that it squares to zero.

Let $\omega\in \Omega_G^*(M))$ and $X\in \g$. 
\begin{align*}
d_C \omega(\Ad_g X)&=d(\omega(\Ad_g X))+\iota((\Ad_g X)^\sharp)\omega(\Ad_g X)\\
&=d(g \omega(X))+\iota(g X^\sharp g\invers) g(\omega(X))\\
&=g d(\omega(X))+g\iota( X^\sharp ) g\invers g(\omega(X))\\&=gd_C \omega 
\end{align*}
Thus $d_C \omega$ is $G$-equivariant. Moreover, we have
\[d_C^2\omega(X)= d^2\omega(X)+d\iota(X)\omega(X)+\iota(X)d\omega(X)+\iota(X)^2\omega(X)=L_X\omega(X)\]
and
\[L_X\omega(X)=\ddt\!\!\!\!\exp(tX)\omega(X)=\ddt\!\!\!\! \omega(\exp(-tX)X\exp(tX))=\ddt\!\!\!\! \omega(X)=0.\]
Thus $d_C$ squares to zero, i.e., it is a boundary operator.
\end{proof}

In the special case of $M=pt$, i.e., of a single point, the Cartan algebra reduces to the algebra of \emph{invariant symmetric polynomials} \index{invariant symmetric polynomials}\symindex[i]{I^k(G)}
\[I^k(G)=((S^*(\g\dual)\otimes\Omega^*(pt))^G)^k=(S^k(\g\dual))^G.\]

\subsection{Getzlers resolution}\label{sec:Getzler}
In order to investigate cohomology of actions of non-compact groups, Ezra Getzler \cite[Section 2]{Getzler} defines a bar-type resolution of the Cartan complex. We will apply his ideas slightly different: The complex defined by Getzler will allow us to compare equivariant integral cohomology (defined via the simplicial manifold) with equivariant cohomology defined by the Cartan model.

Let, as before, a Lie group $G$ act on a smooth manifold $M$ from the left. Define $\C$-vector spaces $C^p(G,S^*(\g\dual)\otimes \Omega^*(M))$ consisting of smooth maps, from the $p$-fold Cartesian product
\[G^p\to S^*(\g\dual)\otimes \Omega^*(M),\]
to the space of polynomial maps from $\g$ to differential forms on $M$. We give these groups a bigrading: The horizontal grading is the one of $S^*(\g\dual)\otimes \Omega^*(M)$ defined above and the vertical grading is  $p$. The Cartan boundary operator $d+\iota$ now induces a map $(-1)^p(d+\iota)$, which increases the horizontal grading by 1 in any row. As we are not restricted to the $G$-invariant part of $S^*(\g\dual)\otimes \Omega^*(M)$, this map will not square to zero, but \[\left((-1)^p(d+\iota)\right)^2=d\iota+\iota d=L\] is the Lie derivative (see e.g. \cite[Proposition 1.121]{GHL}). In vertical direction, there is a differential
\[\bar d\from C^k(G,S^*(\g\dual)\otimes \Omega^*(M))\to C^{k+1}(G,S^*(\g\dual)\otimes \Omega^*(M))\]
defined by
\begin{multline*}
 (\bar d f)(g_0,\dots,g_k|X)\defeq f(g_1,\dots,g_k|X)+ \sum_{i=1}^k (-1)^i f(g_0,\dots,g_{i-1}g_i,\dots,g_k|X)\\+(-1)^{k+1}g_kf(g_0,\dots,g_{k-1}|\Ad(g_k\invers)X)
\end{multline*}
for $g_0,\dots,g_k\in G$ and $X\in \g$.

Note, in particular, that the kernel of
\[\bar d\from C^0(G,S^*(\g\dual)\otimes \Omega^*(M))\to C^1(G,S^*(\g\dual)\otimes \Omega^*(M))\] 
is exactly $\Omega_G^*(M)$. Moreover, in case of a discrete Group $G$, $\g=0$ and thus one checks, that 
\[C^p(G,S^*(\g\dual)\otimes \Omega^*(M))=C^p(G,\Omega^*(M))= \Omega^{p,*}(G^\bullet\times M)\] and $\bar d$ is equal to $\del$.

In the case of a compact Lie group, the map $\bar d$ admits a contraction (compare, e.g., \cite[322]{Gomi}):
\begin{lemma}\label{lem:IntGroupGetzler} Integration over the group, with respect to a right invariant probability measure, defines a map 
\begin{align}
\int_G:C^p(G,S^*(\g\dual)\otimes \Omega^*(M))&\to C^{p-1}(G,S^*(\g\dual)\otimes \Omega^*(M)) \label{eq:IntGroupGetzler}\\
\left(\int_G f\right)(g_1,\dots,g_{p-1},m)&=(-1)^i\int_{g\in G} f(g,g_1,\dots,g_{p-1},m)dg\nonumber
\end{align}
such that $\bar d \int_G f=f$ if $\bar d f=0$.
\end{lemma}
\begin{proof}This is proven by a direct calculation:
 \begin{align*}&\left(\bar d \int_G \omega \right)(g_1,\dots,g_p,m)\\
&=\left(\int_G f\right)(g_2,\dots,g_p|X)+ \sum_{i=2}^p (-1)^i \left(\int_G f\right)(g_1,\dots,g_{i-1}g_i,\dots,g_p|X)\\&\quad +(-1)^{p+1}g_p\left(\int_G f\right)(g_1,\dots,g_{p-1}|\Ad(g_p\invers)X)\\
&=\int_G f(g,g_2,\dots,g_p|X)dg+ \sum_{i=2}^p (-1)^i \int_G f(g,g_1,\dots,g_{i-1}g_i,\dots,g_p|X)dg\\&\quad+\int_G g_pf(g,g_1,\dots,g_{p-1}|\Ad(g_p\invers)X)dg\\
&=\int_G \Bigg(f(g,g_2,\dots,g_p|X)+ \sum_{i=2}^p (-1)^i  f(g,g_1,\dots,g_{i-1}g_i,\dots,g_p|X)\\&\quad+(-1)^{p+1} g_pf(g,g_1,\dots,g_{p-1}|\Ad(g_p\invers)X)\Bigg)dg\\
\intertext{Now we apply $\bar d f(g,g_1,\dots,g_p|X)=0$}
&=\int_G \left(f(g_1,\dots,g_p|X)-f(gg_1,\dots,g_p|X)+f(g,g_2,\dots,g_p|X)\right)dg\\
&=f(g_1,\dots,g_p|X)-\int_G f(gg_1,g_2,\dots,g_p|X)dg+\int_G f(g,g_2,\dots,g_p|X)dg\\
&=f(g_1,\dots,g_p|X)
 \end{align*}
\end{proof}

Thus, for compact groups, the vertical cohomology of this bi-graded collection of groups is the Cartan complex. \\[2ex]
One can turn the bi-graded collection $C^p(G,S^*(\g\dual)\otimes \Omega^*(M))$ of groups into a double complex. Therefore Getzler defines another map,
\[\bar{\iota} \from C^p(G,S^l(\g\dual)\otimes \Omega^m(M))\to C^{p-1}(G,S^{l+1}(\g\dual)\otimes \Omega^m(M)),\]
given by the formula
\[(\bar{\iota} f)(g_1,\dots,g_{p-1}|X)\defeq \sum_{i=0}^{p-1}(-1)^i \ddt f(g_1,\dots,g_i,\exp(tX_i),g_{i+1},\dots,g_{p-1}|X),\]
where $X_i=\Ad(g_{i+1}\dots g_{p-1})X$. 
 \begin{lemma}[Lemma 2.1.1. of {\cite{Getzler}}]The map $\bar{\iota}$ has the following properties:
  \[\bar{\iota}^2=0 \text{ and }\bar d\bar\iota+ \bar\iota\bar d= -L.\]
 \end{lemma}
\begin{proof}
 This is shown in \cite{Getzler} by recollection of the sums in the definition of $\bar{\iota}$ and $\bar d$.
\end{proof}

Moreover one obtains:
\begin{lemma}[Corollary 2.1.2. of {\cite{Getzler}}]
 $d_G= \bar d + \bar \iota + (-1)^p(d + \iota)$ is a boundary operator on the total complex $\bigoplus_{p+2q+r=n} C^p(G,S^q(\g\dual)\otimes \Omega^r(M))$.
\end{lemma}
\begin{proof} $d_G$ increases the total index by one, as $\bar d$ increases the first index, $d$ increases the third index, $\iota$ decreases the third, while it is increasing the second index and $\bar\iota$ decreases the first index, while it is increasing the second one.

As $d$ and $\iota$ are equivariant under the $G$-action, they commute with $\bar d$. And as $d$ and $\iota$ only act on the manifold $M$ and not on the group part, the same is true for $\bar\iota$. Thus
 \begin{align*}
  d_G^2&=(\bar d + \bar\iota)^2 +(-1)^p(\bar d + \bar \iota)(d + \iota) +(-1)^{p\pm 1}(d + \iota)(\bar d + \bar \iota)+ (d + \iota)^2\\
&=\bar d\bar \iota+ \bar\iota \bar d + (d\iota+\iota d)\\
&=-L + L = 0.
 \end{align*}
\end{proof}
\begin{remark}
    The reader, who compares this with the original paper of Getzler will note that we changed some signs. It just seems more natural to us in this way. Furthermore Getzler uses some reduced subcomplex, which is, by standard arguments on simplicial modules (compare Proposition 1.6.5 in \cite{Loday}), quasi-isomorphic to the full complex, which we have taken.
\end{remark}
One can check that \[\left(\left(\bigoplus_{\substack{p+q=n\\q+r=k}} C^p(G,S^q(\g\dual)\otimes \Omega^r(M))\right)_{\!\!\!n,k}\hspace{-.5em},\bar d+(-1)^p\iota,(-1)^pd+\bar\iota\right)\] is a double complex.  But this point of view will not fit to the construction, which we want to do with this bigraded module later: We want to turn the groups $C^p(G,S^*(\g\dual)\otimes \Omega^*(M))$ into simplicial sheaves on $G^\bullet\times M$.
\begin{defn}
    A \emphind{simplicial homotopy cochain complex} of modules is a triple $(M^{\bullet,*}\!,f,s)$, where $M^{\bullet,*}$ is a $\Z$-graded simplicial module, $f$ is a map of simplicial modules, which increases the degree by one and $s$ is a simplicial zero homotopy of $f^2$ which commutes with $f$ and squares to zero, i.e.,\[s\del+\del s=-f^2,\quad sf=fs, \text{ and } s^2=0.\]
\end{defn}
\begin{example}
Observe that
\[\left(C^\bullet(G,S^*(\g\dual)\otimes \Omega^*(M)),d+\iota,\bar\iota\right)\]
is a simplicial homotopy cochain complex. 
\end{example}
\begin{defn2}
The total complex of a simplicial homotopy cochain complex $(M^{\bullet,*}\!,f,s)$ is the chain complex
\[\left(\left(\bigoplus_{p+q=n}M^{p,q}\right)_{\!\!\!n}\!\!,\,\del+s+(-1)^p f\right).\]
\end{defn2}
\begin{proof}
    We have to check that $\del+s+(-1)^p f$ defines a boundary map. Therefore calculate 
\begin{align*}(\del+s+(-1)^p f)^2&=\del^2+s^2+\del s+ s\del+(-1)^p (\del+s) f+(-1)^{p-1} f (\del+s)+f^2\\&=s\del+\del s+f^2\\&=0.\end{align*}
\end{proof}
Observe that the total complexes of the interpretation of $C^\bullet(G,S^*(\g\dual)\otimes \Omega^*(M))$ as double complex and as simplicial homotopy cochain complex coincide. Moreover, note for our applications later, that in the first column, of the double complex interpretation and the degree zero part of the interpretation as simplicial homotopy cochain complex are equal. In formulas this means
\[\left(\bigoplus_{\substack{p+q=n\\q+r=k}} C^p(G,S^q(\g\dual)\otimes \Omega^r(M))\right)_{\!\!\!n,0}\hspace{-.5em}=\;C^n\left(G,S^0(\g\dual)\otimes \Omega^0(M)\right).\]

\subsection{A quasi-isomorphism}\label{sec:quasiiso}
In this section, we will discuss a map defined in \cite[Section 2.2.]{Getzler}. It will relate the complex $C^*(G,S^*(\g\dual)\otimes \Omega^*(M))$ from the last section to the double complex $\Omega^*(G^\bullet\times M)$, which consists in degree $(p,q)$ of $q$-forms on $G^p\times M$ with horizontal boundary map $d=d_{G^p}+d_M$ and vertical boundary map $\del$ from the simplicial manifold structure. Thus we have an explicit identifications of chains in the one complex with chains in the other complex. This will allow us to compare our definition of equivariant differential cohomology (Section \ref{sec:newdef}) with definitions given before.

\begin{defn}[Def. 2.2.1. of \cite{Getzler}] \label{def:Getzlermap} The map
\[\mathcal J\from \Omega^*(G^p\times M)\to \bigoplus_{l=0}^p C^l(G,S^*(\g\dual)\otimes \Omega^*(M))\]\symindex[j]{\mathcal J} \index{Getzler's map}
is defined by the formula
 \[\mathcal J(\omega)(g_1,\dots,g_l|X)\defeq\sum_{\pi\in S(l,p-l)} \sgn(\pi)\left(i_\pi\right)^*\left(\iota_{\pi(l+1)}(X^{(\pi)}_{l+1})\dots\iota_{\pi(p)}(X^{(\pi)}_p)\omega\right).\]
Here $S(l,p-l)$ is the set of shuffles, i.e., permutations $\pi$ of $\{1,\dots,p\}$, satisfying
\[\pi(1)<\dots<\pi(l) \text{ and } \pi(l+1)<\dots<\pi(p),\]
$X^{(\pi)}_j=\Ad(g_{m}\dots g_{l})X$, where $m$ is the least integers less than $l$, such that $\pi(j)<\pi(m)$,
$\iota_j$ means, that the Lie algebra element should be a tangent vector at the $j$-th copy of $G$,
and $i_\pi\from G^l\times M\to G^p\times M$ is the inclusion $x\mapsto (h_1,\dots,h_p,x)$ with \[h_j=\begin{cases}g_m&\text{if } j=\pi(m), 1\leq m\leq l\\ e\in G &\text{otherwise,}\end{cases}\] which is covered by the bundle inclusion $TM\to T(G^p\times M)$.
\end{defn}
Observe that the image of $\omega$ under $\mathcal J$ does only depend on the zero form part and, in direction of any copy of $G$, on the one form part at the identity $e\in G$.

The next Lemma -- which is mainly a citation of \cite[Lemma 2.2.2.]{Getzler}, but with signs corrected -- shows, that the map $\mathcal J$ can be interpreted as a map of double complex.
\begin{lemma}The map $\mathcal J$ respects the boundaries with the correct sign, i.e.,
 \[\mathcal J \circ \del = \left(\bar d+(-1)^p\iota\right) \circ\mathcal J\] and, after decomposing $d=d_G+d_M$ with respect to the Cartesian product $G^p\times M$
 \[\mathcal J \circ \left(\left(-1\right)^p d_M \right)= (-1)^{p'} d \circ \mathcal J \text{ and }\mathcal J \circ \left(-1^p\right)d_G = \bar\iota \circ \mathcal J,\]
where $p$ is the simplicial degree before and $p'$ the simplicial degree after application of the map $\mathcal J$,
\end{lemma}
\begin{proof}The proof is given in \cite[Lemma 2.13]{Ich1}. The idea is, to checks the terms type by type.\end{proof}

Moreover, the map $\mathcal J$ induces an isomorphism in the cohomology of the associated total complexes.
\begin{theorem}[Theorem 2.2.3. of \cite{Getzler}]
 $\mathcal J$ is a quasi-isomorphism.
\end{theorem}

\section{The definition of equivariant differential cohomology}
The are several attempts to a definition \cite{Uribe, Gomi, Ortiz}. The most elaborated one is given by Kiyonori Gomi in \cite{Gomi}, where he defines equivariant smooth Deligne cohomology of a smooth manifold $M$ acted on by a Lie group $G$. His investigations (for $G$ a compact group) can be summarized in the following diagram with exact diagonals
\begin{equation}
\begin{tikzcd}
\faktor{\Omega^{n-1}(M,\C)^G}{\im d} \arrow{dr}{a} & & \Omega_\cl^{n}(M,\C)^G\\
&\hat H^n(M,\Z) \arrow{ru}{R}\arrow{dr}{I}\arrow{r}&H^n_G(M,\C)\\
H^{n-1}(G^*\times M,\pi^{-1}\underline{\C/\Z}) \arrow[hook]{ur}{} &&H^n_G(M,\Z). 
\end{tikzcd}
\end{equation}
The subscript $G$ stands for equivariant cohomology and the superscript $G$ for equivariant forms. Gomi defines the maps and shows that the diagonals are exact in the middle. 

From our point of view, this diagram is not satisfactory: On the one hand, one does not have the Bockstein sequence. On the other hand, closed equivariant forms is not, what one expects in the upper right corner, as there indeed exists a map 
\[\Omega_\cl^{n}(M,\C)^G\to H^n_G(M,\C).\]
But this map is in general not surjective, as not every $n$-class in equivariant cohomology is represented by a closed equivariant $n$-form: There are classes represented by (non-zero degree) polynomials $\mathfrak g\to \Omega^*(M)$. As we have seen these are related to the moment map, which plays an important role when discussing equivariant characteristic classes and forms. Thus this information is neglected in Gomis curvature map.

We will start with the case of finite groups, as there is no moment map if the group is discrete and thus everything is much more similar to the non-equivariant case. In particular, the reader who is not used to differential cohomology hopefully will have an easier access in this way. Afterwards we will discuss the definition of Gomi and show how to define a better curvature map $R$, such that one obtains a hexagon with Gomi's definition of equivariant Deligne cohomology in the middle. By this we are motivated to give another definition, which incorporates additional geometric data. The difficulty is in general not to show \emph{that} there is a hexagon, as this follows directly from the way of the definition by ideas of \cite{BNV}. What the discussion is about, is \emph{which} groups sit on the corners of the hexagon.

At the end of this section we will give some remarks on the definitions of \cite{Uribe, Ortiz} for equivariant differential cohomology.

\emph{Notice that we always work with complex valued differential forms for simplicity.} All statements also hold for real forms and real cohomology.
\subsection{The case of finite groups}
Let, in this subsection, $G$ denote a finite group. Thus a de Rham-type complex for equivariant cohomology is given by
\[\Omega^0(M)^G\overset{d}{\to}\Omega^1(M)^G\overset{d}{\to}\Omega^2(M)^G\overset{d}{\to}\dots .\]
To prepare for the case of Lie groups, we will, nevertheless, work on the simplicial manifold $G^\bullet\times M$. For definition of the Deligne complex, we will use a cone construction similar to \cite{Bunke}.

\begin{defn}
    Let $M$ be a $G$-manifold. The \emph{equivariant Deligne complex} \index{equivariant Deligne complex!finite groups} is defined as 
\[\D(n)_{G^\bullet\times M}=\Cone\left(\underline\Integer\oplus \sigma^{\geq n}\Omega^{\bullet,*}\to\Omega^{\bullet,*},(z,\omega)\mapsto \omega-z\right)[-1].\]
\symindex[d]{\D(n)_{G^\bullet\times M}}
\end{defn}
Here $\underline\Integer$ denotes the locally constant simplicial sheaf of $\Z$ and $\Omega^{\bullet,*}$ the cochain complex of simplicial sheaves of complex valued differential forms on the simplicial manifold $G^\bullet\times M$.
\begin{defn}
 Let $G$ be a finite group acting on a smooth manifold M. The \emph{$G$-equivariant differential cohomology} of $M$ is defined to be the hypercohomology 
\[\hat H^n_G(M)\defeq H^n(G^\bullet\times M, \D(n)_{G^p\times M}).\]\symindex[h]{\hat H^n_G(M)}
\end{defn}
Note that $\D(n)_{G^\bullet\times M}$ is a cochain complex of simplicial sheaves, i.e., it induces a double complex of sheaves so that a resolution is given by triple complex. We give here the general idea of how to define the cohomology of such an object. For details see Section \ref{sec:simpman}.

Denote the boundary map of $\D(n)$ by $\frakd$ and pick an injective resolution $I^{\bullet,*,*}$, i.e., for any $p\in\Natural$ there is the following double complex of sheaves on $G^p\times M$ 
\begin{equation*}
\begin{tikzcd}
\underline{\Z} \arrow{r}{\iota} \arrow{d}{\delta}&C^\infty\arrow{r}{-d} \arrow{d}{-\delta} &\Omega^1\arrow{r}{-d} \arrow{d}{\delta}&
\dots\arrow{r}{-d}&\Omega^{n-2}\arrow{r}{-d} \arrow{d}{(-1)^{n-1}\delta}&\Omega^n\oplus\Omega^{n-1}\arrow{r}{(d,\iota\oplus -d)} \arrow{d}{(-1)^{n}\delta}&\dots\\
I^{p,0,0}\arrow{r}{\frakd} \arrow{d}{\delta}&I^{p,1,0}\arrow{r}{\frakd}  \arrow{d}{-\delta} &I^{p,2,0}\arrow{r}{\frakd}  \arrow{d}{\delta}&
\dots\arrow{r}{\frakd} &I^{p,n-1,0}\arrow{r}{\frakd}  \arrow{d}{(-1)^{n-1}\delta}&I^{p,n,0}\arrow{r}{\frakd}  \arrow{d}{(-1)^{n}\delta}&\dots\\
I^{p,0,1}\arrow{r}{\frakd}  \arrow{d}{\delta}&I^{p,1,1}\arrow{r}{\frakd}  \arrow{d}{-\delta} &I^{p,2,1}\arrow{r}{\frakd}  \arrow{d}{\delta}&
\dots\arrow{r}{\frakd} &I^{p,n-1,1}\arrow{r}{\frakd}  \arrow{d}{(-1)^{n-1}\delta}&I^{p,n,1}\arrow{r}{\frakd}  \arrow{d}{(-1)^{n}\delta}&\dots\\
I^{p,0,2}\arrow{r}{\frakd}\arrow{d}  &I^{p,1,2}\arrow{r}{\frakd}\arrow{d}  &I^{p,2,2}\arrow{r}{\frakd}\arrow{d}  &
\dots\arrow{r}{\frakd} &I^{p,n-1,2}\arrow{r}{\frakd}\arrow{d}  &I^{p,n,2}\arrow{r}{\frakd}\arrow{d}  &\dots\\
\vdots&\vdots&\vdots&\vdots&\vdots&\vdots&
\end{tikzcd}
\end{equation*}
The triple complex which calculates the cohomology is 
\[K^{p,q,r}=\left(\Gamma(G^p\times M,I^{p,q,r}),\del +(-1)^p\frakd+(-1)^{p+q}\delta\right).\]

And in the same spirit as \cite{Bunke} and \cite{Gomi} one investigates differential cohomology by the following two short exact sequences
\begin{align}
\label{eq:ses1}0\to\Cone(\sigma^{\geq n}\Omega^{\bullet,*}\overset{\iota}{\to}\Omega^{\bullet,*})[-1]\overset{a}{\to}&\D(n)\overset{I}{\to} \underline\Z\to 0\\
\label{eq:ses2}0\to\Cone(\underline\Integer\overset{-\iota}{\to}\Omega^{\bullet,*})[-1]\to&\D(n)\overset{R}{\to}\sigma^{\geq n}\Omega^{\bullet,*}_\C\to 0
\end{align}
of complexes of simplicial sheaves and the exact triangle
\[
 \D(n)_{G^\bullet\times M}\to \underline\Z\oplus\sigma^{\geq n} \Omega^{\bullet,*}_\C\to \Omega^{\bullet,*}_\C\to \D(n)_{G^\bullet\times M}[1],
\]
which has the following interesting part in its long exact cohomology sequence
\begin{multline}\label{eq:les}
 H^{n-1}(G^\bullet\times M,\underline\Z)\to H^{n-1}(G^\bullet\times M,\Omega^{\bullet,*}_\C)\to H^n(G^\bullet\times M,\D(n))\\
\overset{(I,R)}{\to} H^n(G^\bullet\times M,\underline\Z)\oplus H^n(G^\bullet\times M,\sigma^{\geq n} \Omega^{\bullet,*})\overset{(-\iota,\iota)}{\to} H^n(G^\bullet\times M,\Omega^{\bullet,*}).
\end{multline}

Recall from Section \ref{sec:simpman} that 
\[H^n(G^\bullet\times M,\underline\Z)=H^n(\|G^\bullet\times M\|,\Z)=H^n_G(M,\Z)\]
and 
\[H^n(G^\bullet\times M,\Omega^{\bullet,*})=H^n_G(M,\C).\]
Moreover, we have:
\begin{lemma} 
    \[H^n(G^\bullet\times M,\sigma^{\geq n} \Omega^{\bullet,*})=\Omega^n_\cl(M)^G.\]
\end{lemma}
\begin{proof}
\begin{align*}
H^n(G^\bullet\times M,\sigma^{\geq n} \Omega^{\bullet,*})&=\{\omega\in \Omega^n(M)| \del \omega=0, d\omega=0\} \\
&=\{\omega\in \Omega^n(M)|  \del_0^*\omega=\del_1^*\omega, d\omega=0\} \\
&=\Omega^n_\cl(M)^G,
\end{align*}
as $\del_0\from G\times M\to M$ is the projection to $M$ and $\del_1$ is the action of $G$ on $M$. Thus we obtain closed equivariant differential $n$-forms.
\end{proof}

As $\Omega^{\bullet,*}$ is a resolution of $\underline\C^\delta$, the locally constant sheaf of continuous maps to the field $\C$ with discrete topology, we get a quasi isomorphism 
\[\Cone(\underline\Integer\overset{-\iota}{\to}\Omega^{\bullet,*})\simeq \Cone(\underline\Integer\overset{-\iota}{\to}\underline\Complex^\delta)\simeq \underline{\faktor{\Complex^\delta}{\Integer}}.\]

The long exact cohomology sequence of the exact triangle 
\[\Cone(\underline\Integer\overset{-\iota}{\to}\underline\Complex^\delta)[-1]\to \underline\Z\overset{-\iota}{\to} \underline\C^\delta\to \Cone(\underline\Integer\overset{-\iota}{\to}\underline\Complex^\delta)\]
is the Bockstein sequence, up to a minus sign. Comparing this with \eqref{eq:les} via the inclusion of \eqref{eq:ses1}, results in the follwing commutative diagram:
\[\begin{scaledcd}{.9}
 H^{n-1}_G(M,\Z)\arrow[equal]{d}\arrow{r}& H^{n-1}_G(M,\C)\arrow[equal]{d}\arrow{r}& H^{n-1}_G(M,\C/\Z) \arrow{d}\arrow{r}{-\beta}&H^n_G(M,\Z)\arrow{d}{\id\oplus 0}\arrow{r}{-\iota}& H^n_G(M,\C)\arrow[equal]{d}\\
 H^{n-1}_G(M,\Z)\arrow{r}& H^{n-1}_G(M,\C)\arrow{r}& \hat H^n_G(M,\Z)\arrow{r}{I\oplus R}&H^n_G(M,\Z)\oplus  \Omega^*_{\cl}(M)^G\arrow{r}{(-\iota,\iota)}& H^n_G(M,\C).
\end{scaledcd}\]
For the next steps, recall that one can define, similar to Lemma \ref{lem:IntGroupGetzler}, a vertical contraction of the double complex $\Omega^{\bullet,*}(G^\bullet\times M)$.
\begin{lemma}\label{lem:IntGroup}`Integration' over the group, defines a map 
\begin{align}
\int_G:\Omega_{\C}^{i}(G^k\times M)&\to \Omega_{\C}^{i}(G^{k-1}\times M) \label{eq:IntGroup}\\
\left(\int_G \omega\right)(g_1,\dots,g_{k-1},m)&=(-1)^i\frac{1}{|G|}\sum_G \omega(g,g_1,\dots,g_{k-1},m)\nonumber
\end{align}
such that $\del \int_G(\omega)=\omega$ if $\del \omega=0$.
\end{lemma}
\begin{proof}
 The proof is given by the same calculation as for Lemma \ref{lem:IntGroupGetzler}.
\end{proof}

One has $\Cone(\sigma^{\geq n}\Omega^{\bullet,*}\overset{\iota}{\to}\Omega^{\bullet,*})[-1]\simeq \sigma^{<n}\Omega^{\bullet,*}[-1]$. Thereby, the map in degree $n$, is given from right to left by $\eta\mapsto (d\eta,\eta)$. As differential forms form a fine sheaf, they are themselves their injective resolution. Thus the cohomology group $H^{n-1}(G^\bullet\times M;\Omega^{<n}_\C)$ is given by the hypercohomology of the double complex of vector spaces:
\begin{equation*}
\begin{tikzcd}
\Omega^0(M)\arrow{r}{-d} \arrow{d}{\del} &\dots\arrow{r}{-d}&\Omega^{n-2}(M)\arrow{r}{-d} \arrow{d}{\del}&\Omega^{n-1}(M)\arrow{r} \arrow{d}{\del}&0\\
\Omega^0(G\times M)\arrow{r}{d} \arrow{d}{\del} &\dots\arrow{r}{d}&\Omega^{n-2}(G\times M)\arrow{r}{d} \arrow{d}{\del}&\Omega^{n-1}(G\times M)\arrow{r} \arrow{d}{\del}&0\\
\Omega^0(G^2\times M)\arrow{r}{-d} \arrow{d}{\del} &\dots\arrow{r}{-d}&\Omega^{n-2}(G^2\times M)\arrow{r}{-d} \arrow{d}{\del}&\Omega^{n-1}(G^2\times M)\arrow{r} \arrow{d}{\del}&0\\
\vdots&&\vdots&\vdots&
\end{tikzcd}
\end{equation*}
The contraction of \autoref{lem:IntGroup} implies by the little spectral sequence argument recalled in \autoref{lemma:specseq}, that the inclusion of
\[\Omega^0(M)^G\overset{-d}{\to}\dots\overset{-d}{\to}\Omega^{n-2}(M)^G\overset{-d}{\to}\Omega^{n-1}(M)^G\to 0\]
as first row into the double complex induces an isomorphism in cohomology.
 
Thus we obtain
\[ H^{n-1}(G^\bullet\times M,\Omega^{<n})=\faktor{\left(\Omega^{n-1}(M)\right)^G}{d\left(\Omega^{n-2}(M)^G\right)}\,,\]
and from \eqref{eq:ses1} an induced map 
\[a\colon\faktor{(\Omega^{n-1}(M))^G}{d\left(\Omega^{n-2}(M)^G\right)}\to H^n(G^\bullet\times M,\D(n)).\]
\begin{lemma}
 $R\circ a=d\colon \faktor{(\Omega^{n-1}(M))^G}{d\left(\Omega^{n-2}(M)^G\right)}\to \Omega_{\cl}^{n}(M)^G$
\end{lemma}
\begin{proof}
 Let $\eta\in \faktor{(\Omega^{n-1}(M))^G}{d\left(\Omega^{n-2}(M)^G\right)}$. Recall that the quasi-isomorphism $ \sigma^{<n}\Omega^{\bullet,*}[-1]\simeq\Cone(\sigma^{\geq n}\Omega^{\bullet,*}\overset{\iota}{\to}\Omega^{\bullet,*})[-1]$ is in degree $n$ given by $\eta\mapsto(d\eta,\eta)$. Thus one calculates
\[R(a(\eta))=R(d\eta,\eta)=d\eta,\]
since $R$ comes from the projection to the first summand.
\end{proof}

Collecting these statements, we have proven the following theorem.
\begin{theorem}\label{thm:finitehexagon}Let $G$ be a finite group acting on the smooth manifold $M$, then there is a commutative diagram
\[ 
\begin{tikzcd}[column sep=small]
\phantom{.} &\hspace{-4em}\faktor{\Omega^{n-1}(M)^G}{d\left(\Omega^{n-2}(M)^G\right)} \arrow{dr}{a} \arrow{rr}{d}& & \Omega_{\cl}^{n}(M)^G\arrow{dr}{}&\\
H_G^{n-1}(M,\C)\arrow{ru}{}\arrow{dr}{}&&\hat H_G^n(M,\Z)\arrow{ru}{R}\arrow[two heads]{dr}{I}&&\hspace{-1em}H_G^n(M,\C)\\
&\hspace{-3em}H_G^{n-1}(M,\C/\Z) \arrow[hook]{ur}{} \arrow{rr}{-\beta}&&H_G^n(M,\Z) \arrow{ur}{}&
\end{tikzcd}
\]
whose top line, bottom line and diagonals are exact. 
\end{theorem}
\begin{remark}
 If $G$ is the trivial group this is exactly the hexagon of the non-equivariant case.
\end{remark}

\begin{prop}
 If $G$ acts freely on $M$, then \[\hat H_G^n(M,\Z) =\hat H^n\left(\faktor{M}{G},\Z\right)\]
\end{prop}
\begin{proof}
 Let $Q$ denote the quotient manifold, $q\from M\to Q$ the quotient map and $\{e\}$ the trivial group. The simplicial manifold $\{e\}^\bullet\times Q$ equals $Q$ in any level and all face and degeneracy maps are the identity. This implies that 
\[\hat H^n_{\{e\}}(Q,\Z)=H^n(\{e\}^\bullet\times Q,\D(n))=\hat H^n(Q,\Z),\]
since $\del$ alternately equals $\id$ or the zero map.

Moreover, $q$ and $G\to \{e\}$ induce a smooth simplicial map $G^\bullet\times M\to \{e\}^\bullet\times Q$, whose geometric realization is a fattened, homotopy equivalent, version of $EG\times_G M\to Q$. This map induces homomorphism between the exact lines 
\[\begin{scaledcd}{.9}
 H^{n-1}(Q,\Z)\arrow{d}\arrow{r}& H^{n-1}(Q,\C)\arrow{d}\arrow{r}& \hat H^n(Q,\Z) \arrow{d}\arrow{r}&H^n(Q,\Z)\oplus \Omega^n_{\cl}(Q)\arrow{d}\arrow{r}& H^n(Q,\C)\arrow{d}\\
 H^{n-1}_G(M,\Z)\arrow{r}& H^{n-1}_G(M,\C)\arrow{r}& \hat H^n_G(M,\Z)\arrow{r}&H^n_G(M,\Z)\oplus  \Omega^n_\cl(M)^G\arrow{r}& H^n_G(M,\C)
\end{scaledcd}\]
As the two maps on the left and the right-hand side are isomorphisms, the same is true in the middle by the five lemma.
\end{proof}

How about homotopy invariance in equivariant differential cohomology? It is the same as in the non-equivariant case: It is not, but one can measure the deviation from being so.
\begin{prop}
Let $i_t \colon M \to  [0,1]\times M$ be the inclusion at $t\in [0,1]$, let $G$ act trivially on the interval and let $\hat x\in \hat H_G^n([0,1]\times M,\Z)$, then 
\[i_1^*\hat x-i_0^*\hat x=a\left(\int_{[0,1]\times M/M} R(\hat x)\right).\]
\end{prop}
\begin{proof}
 This is almost verbatim \cite[Prop. 3.28]{Bunke}. As integral cohomology is homotopy invariant, there is $y\in H^n_G(M,\Z)$, such that $I(x)=\pr_M^*y$. By surjectivity of $I$ there is a lift $\hat y\in \hat H_G^n(M,\Z)$, with $I(\hat y)=y$. Hence $I(\hat x - \pr_M^*\hat y)=0$ and thus $\hat x = \pr_M^*\hat y + a(\omega)$ for some $\omega\in (\Omega^{n-1}([0,1]\times M))^G$. Note that $R(\hat x) = R(\pr_M^*\hat y) + d\omega$.
\[i_1^*\hat x-i_0^*\hat x=i_1^*a(\omega)-i_0^*a(\omega)=a\left(\int_{[0,1]\times M/M} d\omega\right)=a\left(\int_{[0,1]\times M/M} R(\hat x)\right).\]
Where the second equality is Stokes theorem, and the third equality follows
as the fiber integral of pullback forms is zero:
\[\int_{[0,1]\times M/M} R(\pr_M^*\hat y)=\int_{[0,1]\times M/M} \pr_M^*(R(\hat y))=0.\]
\end{proof}

\subsection{A version for Lie groups by Gomi}

From now on, let $G$ denote a Lie group. In this section we restate the definition for equivariant smooth Deligne cohomology given in \cite{Gomi} and show how one can define a `curvature' map, which does lead to a differential cohomology hexagon. Gomi notes the lack of his definition himself (see \cite[Lemma 5.9]{Gomi}) and this lemma is the starting point for our alteration.\\[1ex]
Combining the ideas of Gomi with the cone construction, we reformulate the definition for the equivariant Deligne complex for Lie groups by Gomi:
\begin{defn} Let $M$ be a $G$-manifold for a Lie group $G$. The equivariant Deligne complex in degree $n$ is defined as
\[\DG(n)_{G^\bullet\times M}=\Cone(\underline\Integer\oplus\mathcal F^1_n\Omega^{\bullet,*}\to\Omega^{\bullet,*},(z,\omega)\mapsto \omega-z)[-1].\]
\symindex[d]{\DG(n)_{G^\bullet\times M}}\symindex[F]{\mathcal F^1_n\Omega^*_\Complex}
Here $\mathcal F^1_n\Omega^*_\Complex$ is the simplicial sub-sheaf achieved from the simplicial sheaf of differential forms on $G^\bullet\times M$ by imposing the following conditions: in simplicial level zero, i.e., on $M$, forms shall have at least degree $n$ and on any other level the differential form degree on the $G$-part is at least $1$, if the total form degree is less then $n$. 
\end{defn}

In particular, if $G$ is discrete, this is the same complex as in the last section, because on discrete groups $G$ there are no positive degree differential forms.
\begin{defn}
 Let $G$ be a Lie group acting on a smooth manifold M. The $G$-\emphind{equivariant differential cohomology} \symindex[H]{\hat H^n_G(M)} of $M$ is defined to be the hypercohomology 
\[\hat H^n_G(M)\defeq H^n(G^\bullet\times M, \DG(n)_{G^p\times M}).\]
\end{defn}
We investigate equivariant differential cohomology for Lie groups with the same methods as for finite groups namely with the following two short exact sequences
\begin{align}
\label{eq:Lieses1}0\to\Cone(\F^1_n\Omega^{\bullet,*}\overset{\iota}{\to}\Omega^{\bullet,*})[-1]\overset{a}{\to}&\DG(n)\overset{I}{\to} \underline\Z\to 0\\
\label{eq:Lieses2}0\to\Cone(\underline\Integer\overset{-\iota}{\to}\Omega^{\bullet,*})[-1]\to&\DG(n)\overset{R}{\to}\mathcal F^1_n\Omega^{\bullet,*}\to 0
\end{align}
of complexes of simplicial sheaves and also with the exact triangle
\[
 \DG(n)_{G^\bullet\times M}\to \underline\Z\oplus\mathcal F^1_n\Omega^{\bullet,*}\to \Omega^{\bullet,*}\to \DG(n)_{G^\bullet\times M}[1],
\]
which has the following interesting part in its long exact cohomology sequence
\begin{multline*}
 H^{n-1}(G^\bullet\times M,\underline\Z)\to H^{n-1}(G^\bullet\times M,\Omega^{\bullet,*})\to H^n(G^\bullet\times M,\DG(n))\\
\overset{(I,R)}{\to} H^n(G^\bullet\times M,\underline\Z)\oplus H^n(G^\bullet\times M,\F^1_n \Omega^{\bullet,*})\overset{(-\iota,\iota)}{\to} H^n(G^\bullet\times M,\Omega^{\bullet,*}).
\end{multline*} 
As before one has (\cite[Prop. 5.15 and Prop. 6.1]{Dupont}) 
\[H^n(G^\bullet\times M,\underline\Z)=H^n_G(M,\Z),\qquad H^n(G^\bullet\times M,\Omega^{\bullet,*})=H^n_G(M,\C)\]
and 
\[\Cone(\underline\Integer\overset{-\iota}{\to}\Omega^{\bullet,*})\simeq \underline{\faktor{\Complex^\delta}{\Integer}}.\]

Thus, the only things left to discuss, are the differential form sheaves on the left in \eqref{eq:Lieses1} and on the right in \eqref{eq:Lieses2}.
\begin{lemma}[Lemma 4.5 of {\cite{Gomi}}] \label{lemma:LieForms} Let $G$ be a compact Lie group, then
 \[H^n(G^\bullet\times M,\Cone(\F^1_n\Omega^{\bullet,*}\overset{\iota}{\to}\Omega^{\bullet,*})[-1])=\faktor{\Omega^{n-1}(M)^G}{d(\Omega^{n-2}(M)^G)}\]
and
\[H^{n+1}(G^\bullet\times M,\Cone(\F^1_n\Omega^{\bullet,*}\overset{\iota}{\to}\Omega^{\bullet,*})[-1])=0.\]
\end{lemma}
\begin{proof}Since all sheaves of differential forms occuring are fine, we can turn to the global sections. Moreover notice, that if \[\pi\from G^\bullet\times M \to \{e\}^\bullet\times M\]
is the map of simplicial manifolds induced by the unique map $G\to\{e\}$, then there is a quasi-isomorphism \[\Cone(\F^1_n\Omega^{\bullet,*}\overset{\iota}{\to}\Omega^{\bullet,*})\to \pi^*\Omega^{\bullet,<n}.\]

Further notice that, in analogy to \eqref{eq:IntGroup}, integration over the group with respect to a right invariant probability measure induces a contraction of the double complex in vertical direction. Hence the cohomology of the double complex 
\[\left(\pi^*\Omega^{\bullet,<n}(G^\bullet\times M),d,\del\right)\]
coincides with the cohomology of $\left(\Omega^{<n}(M)^G,d\right)$, what implies the two assertions.
\end{proof}
From \eqref{eq:Lieses1} we obtain an induced map 
\[a\colon\faktor{(\Omega^{n-1}(M))^G}{d\left(\Omega^{n-2}(M)^G\right)}\to H^n(G^\bullet\times M,\DG(n)).\]
\begin{lemma}
 $R\circ a=d+\del\colon \faktor{(\Omega^{n-1}(M))^G}{d\left(\Omega^{n-2}(M)^G\right)}\to \Omega_{\cl}^{n}(M)^G\oplus \mathfrak{g}\dual\otimes\Omega^{n-2}(M)$
\end{lemma}
\begin{proof}
 Let $\eta\in \faktor{(\Omega^{n-1}(M))^G}{d\left(\Omega^{n-2}(M)^G\right)}$. The quasi-isomorphism applied in the proof of Lemma \ref{lemma:LieForms} together with the inclusion of the invariant forms into the simplical chain complex is defined, such that $a(\eta)=(0,(\dots,0,\del \eta,d\eta),(\dots0,\eta))$. Moreover $R$ is the projection to the tuple of forms in the middle. Thus we obtain the assertion.
\end{proof}
Collecting the statements, shown above, yields to the following theorem.
\begin{theorem}\label{thm:GomiHexagon1}Let $G$ be a compact Lie group acting from the left on the smooth manifold $M$. Then there is the following commutative diagram
\begin{equation} \label{LieHexagon}
\begin{scaledcd}{.9}
\phantom{.} &\hspace{-5em}\faktor{(\Omega^{n-1}(M))^G}{d\left(\Omega^{n-2}(M)^G\right)} \arrow{dr}{a} \arrow{rr}{d+\del}& & H^n(G^\bullet\times M,\F_n^1\Omega^*)\arrow[two heads]{dr}{}&\\
H_G^{n-1}(M,\C)\arrow{ru}{}\arrow{dr}{}&&\hspace{-1em}\hat H_G^n(M,\Z) \hspace{-1em}\arrow{ru}{R}\arrow[two heads]{dr}{I}&&\hspace{-1em}H_G^n(M,\C),\\
&\hspace{-3em}H_G^{n-1}(M,\C/\Z) \arrow{ur}{} \arrow{rr}{-\beta}&&H_G^n(M,\Z) \arrow{ur}{}&
\end{scaledcd}
\end{equation}
where the top line, the bottom line and the diagonals are exact. 
\end{theorem}
\begin{remark}
 Parts of this diagram are due to Gomi (\cite{Gomi}), but, as he -- partially -- defined maps to different groups in the corners, he did not achieve the entire hexagon. The curvature map of Gomi can be recovered by combining the curvature map $R$, given above in the hexagon, with the map
 \[H^n(G^\bullet\times M,\F_n^1\Omega^*)\to \Omega^n_\cl(M)^G,\]
 induced from projecting a cocycle $\bigoplus_{i=0}^n \Omega^{n-1}(G^i\times M)\ni(\omega_i)\mapsto \omega_0$ to the invariant form part.

 If $G$ is a discrete group, this reduces to the diagram of \autoref{thm:finitehexagon}. If $G$ is non-discrete and acting freely on $M$, such that the quotient space is a manifold, one would like to compare equivariant differential cohomology with differential cohomology of the quotient. In general, one can not expect, that $\hat H_G^n(M,\Z)=\hat H^n(M/G,\Z)$ as  $\left(\Omega^{n-1}(M)\right)^G$ is different from $\Omega^{n-1}(M/G)$. To see this in a very explicit example, take $M=G$, then, in degree $n=2$,  $\Omega^{n-1}(M)^G=\Omega^{1}(G)^G=\mathfrak g\dual$, but $\Omega^{n-1}(M/G)=\Omega^{1}(pt)=0$.

 Moreover, one can not expect, that the map $H_G^{n-1}(M,\C/\Z)\to \hat H_G^n(M,\Z)$ is injective as in the discrete case, because $H^{n-1}(G^\bullet\times M,\F_n^1\Omega^*)$ will not vanish in general. To see this, take the following example for any positive dimensional Lie group $G$:
\[
H^{2}(G^\bullet\times M,\F_3^1\Omega^*)=\ker\left(d+\del\colon \F^1\Omega^1(G\times M)\to \Omega^2(G\times M)\oplus\Omega^1(G\times G\times M)\right)
\]
If $\omega\in \Omega^1(G\times M)$ has form degree one on $G$, then $\del \omega=0$ means that for any $g_1,g_2 \in G$, $m\in M$ and any vector field $X=X_1+X_2+X_M$, decomposed into the tangent direction of the first copy of $G$, the second copy of $G$ and $M$, one has
\begin{align}
 0&=(\del \omega)(g_1,g_2,m)[X]\nonumber\\
&=\omega(g_2,m)[X_2]-\omega(g_1g_2,m)[X_1g_2+g_1X_2]+\omega(g_1,g_2m)[X_1]\label{eq:exampleform}
 \end{align}
Taking $X_1=0$ this implies, that actually $\omega=f\in C^\infty(M,g^\vee\otimes \C)$. Moreover, taking $X_2=0$ in \eqref{eq:exampleform}, we obtain $Ad_g\circ f= L_g^* f$ for any $g\in G$. Finally, since $d\omega=0$, one has $d_M f=0$. Hence 
\[H^{2}(G^\bullet\times M,\F_3^1\Omega^*)=\map (\pi_0(M),g^\vee)\neq\emptyset.\]
\end{remark}
\begin{example} 
When constructing characteristic classes, the cohomology of the classifying space is highly interesting. Let $G$ be a group and $EG\to BG$ the universal bundle. Then for cohomology with any coefficient group one has
\[H^*(BG)=H^*(EG/G)=H^*_G(EG)=H^*_G(pt),\]
where $pt$ is the manifold consisting of a single point. Hence our question is: What is $\hat H^*_G(pt,\Z)$?

In this case the hexagon becomes
\begin{equation} 
\begin{tikzcd}
\phantom{.} &\Omega^{n-1}(pt) \arrow{dr}{a} \arrow{rr}{d+\del}& & H^n(G^\bullet\times pt,\F_n^1\Omega^*)\hspace{-2em}\arrow{dr}{}&\\
H^{n-1}(BG,\C)\arrow{ru}{}\arrow{dr}{}&&\hspace{-1em}\hat H_G^n(pt,\Z) \hspace{-1em}\arrow{ru}{R}\arrow[two heads]{dr}{I}&&\hspace{-1em}H^n(BG,\C).\\
&\hspace{-3em}H^{n-1}(BG,\C/\Z) \arrow{ur}{} \arrow{rr}{-\beta}&&H^n(BG,\Z) \arrow{ur}{}&
\end{tikzcd}
\end{equation}
Hence $\hat H_G^n(pt,\Z)= H_G^n(pt,\Z)$ if $n\neq 1$ and, as $H^1(G^\bullet\times pt,\F_1^1\Omega^*)=0$, we get $\hat H^1_G(pt)=H_G^0(pt,\C/\Z)=\C/\Z$ if $G$ is connected. Maybe one wonders whether this $\C/\Z$ yields some characteristic class like information. The answer is: Pulling back an element of $\C/\Z$ via the classifying map of some principal $G$ bundle, just gives a constant function on the base space.
\end{example}

In Section \ref{sec:Cartan} we defined the Cartan complex 
\[(d+\iota)_n \from \Omega_G^n(M) \to \Omega^{n+1}_G(M)\] 
which calculates equivariant cohomology, where $\Omega^n_G(M)=\left(\left(S^*(\g\dual)\otimes \Omega^*(M)\right)^G\right)^n$. We want to compare the group $H^n(G^\bullet\times M,\F_n^1\Omega^*)$ in the upper right corner of \eqref{LieHexagon} with the Cartan model.
\begin{prop}\label{pro:f1niso}
There is a natural isomorphism
\[H^n(G^\bullet\times M,\F_n^1\Omega^*) \to \raisebox{1ex}{$\ker(d+\iota)_n$}\mkern-8mu\raisebox{-1ex}{\scalebox{2}{$\diagup$}}\mkern-15mu\raisebox{-2.5ex}{$(d+i)\left(\bigoplus\limits_{k=1}^{n/2} \left(S^k(\g\dual)\otimes \Omega^{n-1-2k}(M)\right)^G\right).$}\]\end{prop}
\begin{proof}
 In Section \ref{sec:quasiiso} we defined a quasi-isomorphism 
\[\mathcal J\from \Omega^*(G^p\times M)\to \bigoplus_{l=0}^p C^l(G,S^*(\g\dual)\otimes \Omega^*(M)).\]
Let 
\[X^{l,k,m}=\begin{cases}
	     0 &\text{ if } k=0 \text{ and } m<n,\\
             C^l(G,S^k(\g\dual)\otimes \Omega^m(M) &\text{otherwise.}
            \end{cases}\]
The double complex $(X^{\bullet,(2*+*)},d+\iota+\bar d+\bar\iota)$ is a subcomplex of $C^\bullet(G,S^*(\g\dual)\otimes \Omega^*(M))$: One has to check, that the inclusion commutes with boundaries. By the way $X$ is defined, the only reason for which it is maybe not a subcomplex, could arise from the maps which are turned into zero maps, as they map to the zero space. Thus, the problem can only come from maps lowering indices, namely $\iota$ and $\bar\iota$, but these two raise the second index, hence there image does not lie in one the spaces $X^{l,0,m}$, with $m<n$. 

From the definition of $\mathcal J$ one checks that 
\[\mathcal J(\F_n^1\Omega^*(G^\bullet\times M))\subset X^{\bullet,*,*}.\]
Moreover, $\mathcal J$ is the identity on those forms, which have vanishing degree on the group part and \[H^{n-1}\left(\faktor{C^\bullet(G,S^*(\g\dual)\otimes \Omega^*(M))}{X^{\bullet,2*,*}}\right)=
\faktor{\Omega^{n-1}(M)^G\!\!}{d\left(\Omega^{n-2}(M)^G\right)}\] by integration over the first copy of $G$ (compare Lemma \ref{lem:IntGroup}). Hence, $\mathcal J$ and the inclusion of the Cartan complex into Getzler's resolution induce the following commutative diagram with exact rows
\[\begin{scaledcd}{0.8}
H_G^{n-1}(M,\C)\arrow{r}{}\arrow{d}{\mathcal J_*} &\faktor{(\Omega^{n-1}(M))^G}{d\left(\Omega^{n-2}(M)^G\right)} \arrow{d}{\id} \arrow{r}{d+\del}& H^n(G^\bullet\times M,\F_n^1\Omega^*)\arrow{r}{}\arrow{d}{\mathcal J_*}&H_G^n(M,\C)\arrow{r}\arrow{d}&0\\
H_G^{n-1}(M,\C)\arrow{r}{} &\faktor{(\Omega^{n-1}(M))^G}{d\left(\Omega^{n-2}(M)^G\right)} \arrow{r}{d+\iota}& H^n(X^{*,(2*+*)})\arrow{r}{}&H_G^n(M,\C)\arrow{r}&0\\
H_G^{n-1}(M,\C)\arrow{r}{}\arrow{u}{} &\faktor{(\Omega^{n-1}(M))^G}{d\left(\Omega^{n-2}(M)^G\right)} \arrow{u}{\id} \arrow{r}{d+\iota}&\faktor{\ker(d+\iota)_n}{\sim} \arrow{u}\arrow{r}{}&H_G^n(M,\C)\arrow{u}\arrow{r}&0
\end{scaledcd}\]
where $\faktor{\ker(d+\iota)_n}{\sim}$ should denote the right-hand side of the assertion. By the five lemma this diagram shows that there is the isomorphism as claimed.
\end{proof}

The discussion of this section thus manifests in the following alteration of Theorem \ref{thm:GomiHexagon1}.
\begin{theorem}For any compact Lie group acting on a smooth manifold $M$, there is the commutative diagram
\begin{equation} \label{LieHexagon2}
\begin{scaledcd}{.8}
\phantom{.} &\hspace{-5em}\faktor{(\Omega^{n-1}(M))^G}{d\left(\Omega^{n-2}(M)^G\right)} \arrow{dr}{a} \arrow{rr}{d+\iota}& & \hspace{-.2em} \faktor{\ker(d+\iota)_n}{(d+i)\left(\bigoplus_{k=1}^{n/2} S^k(\g\dual)\otimes \Omega^{n-1-2k}(M)\right)^G}\hspace{-6em}\arrow[two heads]{dr}{}&\\
H_G^{n-1}(M,\C)\arrow{ru}{}\arrow{dr}{}&&\hspace{-1em}\hat H_G^n(M,\Z) \arrow{ru}{R}\arrow[two heads]{dr}{I}&&\hspace{-1em}H_G^n(M,\C)\\
&\hspace{-3em}H_G^{n-1}(M,\C/\Z) \arrow{ur}{} \arrow{rr}{-\beta}&&H_G^n(M,\Z) \arrow{ur}{}&
\end{scaledcd}
\end{equation}
whose top line, bottom line and diagonals are exact.
\end{theorem}
\subsection{The new version for Lie groups}\label{sec:newdef}
In the last section, the geometric refinement was done only with respect to the manifold. In this section, we will give another solution, where one enriches the equivariant cohomology by all closed Cartan forms.

Therefore we want to use the model for equivariant cohomology defined by Getzler, which we introduced in Section \ref{sec:Getzler}. As noted there, this model is not a cochain complex of simplicial modules, but only a simplicial homotopy cochain complex. To proceed as in the previous Section and do a similar cone-construction, we first have to investigate the algebraic structure of simplicial homotopy cochain complexes in more detail.

\begin{defn}
    A \emphind{simplicial sheaf homotopy cochain complex} of modules on a simplicial manifold $M_\bullet$ is a triple $(\F^{\bullet,*},f,s)$, where $\F^{\bullet,*}$ is a $\Z$-graded simplicial sheaf of modules on $M_\bullet$, which is bounded from below\footnote{This means, there is an Integer $k$ such that each $\F^{p,q}=0$ if $q<k$.}, $f$ is a map of simplicial sheaves, which increases the $\Z$-grading by one and $s$ is a simplicial zero homotopy of $f^2$, i.e., in simplicial degree $p$, $s=(s_i)_{i=0,\dots,p-1}$, where \[s_i\from \sigma_i\invers \F^{p,q}\to \F^{p-1,q+1},\quad i=0,\dots,p-1\]
are maps of sheaves, such that the simplicial relations of degeneracy maps hold, $s$ commutes with $f$ and 
\begin{align*}
    s_p\circ\tilde{\del}_{p+1}&=-f^2\from \left(\sigma_p\invers\left(\del_{p+1}\invers\F^{p,q}\right)\right)=\F^{p,q}\to \F^{p,q+1}\\
	s_i\circ\tilde{\del}_j&=\begin{cases}\tilde{\del}_{j}\circ s_{i-1}&\text{if } i< j\\
	                            \tilde{\del}_{j-1}\circ s_{i}&\text{if } i>j+1
	                        \end{cases}\\
	s_j\circ\tilde{\del}_j=s_j\circ\tilde{\del}_{j+1}
	s_0\circ\tilde{\del}_0&=0.
\end{align*}
A \emph{morphism} of simplicial sheaf homotopy cochain complex is a map of the simplicial sheaves, which respects the grading and commutes with both the `boundary map' $f$ and the zero homotopy.
\end{defn}
\begin{defn2}
    Let $w\from (\F^{\bullet,*},f,s)\to(\widetilde{\F}^{\bullet,*},\tilde{f},\tilde{s})$ be a morphism of simplicial homotopy cochain complex. The \emphind{cone} of $w$ is the simplicial sheaf homotopy cochain complex
\[\Cone(w)\defeq\left(\left(\F^{\bullet,k+1}\oplus\widetilde{\F}^{\bullet,k}\right)_{k\in\Natural},\begin{pmatrix}-f&-w\\0&\tilde f\end{pmatrix},\begin{pmatrix}s&0\\0&\tilde s\end{pmatrix} \right).\]
\end{defn2}
\begin{proof}
 The only point, which is worth to check, is the relation between the `boundary map' and the homotopy:
\begin{multline*}
    -\begin{pmatrix}-f&-w\\0&\tilde f\end{pmatrix}^2=-\begin{pmatrix}f^2&fw-w\tilde f\\0&\tilde f^2\end{pmatrix}=\begin{pmatrix}-f^2&0\\0&-\tilde f^2\end{pmatrix}\\=\begin{pmatrix}s\del+\del s&0\\0&\tilde s\del+\del\tilde s\end{pmatrix}=\begin{pmatrix}s&0\\0&\tilde s\end{pmatrix}\del+\del\begin{pmatrix}s&0\\0&\tilde s\end{pmatrix}.
\end{multline*}
\end{proof}
We are now going to define the cohomology of a simplicial sheaf homotopy cochain complex $(\F^{\bullet,*},f,s)$ using a \Cech{} model. Let $\mathcal U^\bullet$ be a simplicial cover of the simplicial manifold $M_\bullet$. This defines for each $q$ a resolution of the simplicial sheaf $\F^{\bullet,q}$ (compare Section \ref{sec:simpcechcoho}) 
\[\check C^{\bullet,q,*}(\mathcal U^\bullet,\mathcal F^{\bullet,k})\]
with \Cech{} boundary map $\delta$. The properties of the simplicial cover imply, that $\del$ and $s$ restrict to the \Cech{} groups. Hence, on the total complex of this triple graded collection of modules, we have a boundary map
\[\left(\bigoplus_{p+q+r=n} \check C^{p,q,r},\del+s+(-1)^p f+(-1)^{p+q}\delta\right),\] where $\del$ and $s$ are the alternating sums over the maps $\tilde\del_i$ and $s_i$ respectively.

Thus we can define $\check H(\mathcal U^\bullet,(\F^{\bullet,*},f,s))$ to be the cohomology of this cochain complex. As for classical \Cech{} cohomology, refinements of the simplicial cover induce homomorphisms of the associated cohomology theories. Thus we define
\[\check H(M_\bullet,(\F^{\bullet,*},f,s))=\lim_{\mathcal U^\bullet}\check H(\mathcal U^\bullet,(\F^{\bullet,*},f,s))\]
to be the limit over all refinements of open covers.

If the simplicial sheaf homotopy cochain complex $(\F^{\bullet,*},f,s)=(\F^{\bullet,*},d,0)$ actually is a cochain complex of simplicial sheaves, the total complex of the \Cech{} resolution of both types (compare Section \ref{sec:simpcechcoho}) coincides, and hence the cohomology defined here, coincides with the simplicial sheaf cohomology. Moreover, if the sheaves of $(\F^{\bullet,*},f,s)$ are fine, then the \Cech{} direction contracts by the standard argument and the cohomology of $(\F^{\bullet,*},f,s)$ is the cohomology of the total complex $(\oplus_{p+q=n}\F^{p,q}(G^p\times M),(-1)^p f+s+\del)$.
\\[2ex]
Now, turn to our specific case, i.e., we would like to find a simplicial sheaf homotopy cochain complex $\mathcal C^{\bullet}=\mathcal C^{\bullet,*}$ consisting of fine sheaves, such that its global sections are given by $C^\bullet(G,S^*(\g\dual)\otimes \Omega^*(M))$. 

The map $\pi\from G^\bullet\times M \to \{e\}^\bullet\times M, (g_1,\dots,g_p,m)\mapsto (g_1\dots g_pm)$ is a morphism of simplicial manifolds. $S^*(\g\dual)\otimes \Lambda^* T\dual\! M$ is a bundle over $M$, with left action of $G$ on $M$, the induced action on the cotangent bundle and coadjoint action on the polynomial, whose global sections are $S^*(\g\dual)\otimes \Omega^*(M)$. We can interpret this bundle as simplicial bundle on the simplicial manifold $\{e\}^\bullet\times M$, with all face and degeneracy maps being the identity. The global sections of the pullback bundle $\pi^*(S^*(\g\dual)\otimes \Lambda^* T\dual\! M)$ in simplicial level $p$ are $C^p(G,S^*(\g\dual)\otimes \Omega^*(M))$, thus take for $U\subset G^p\times M$ open
\[\mathcal C^p(U)\defeq\Gamma(U,(\pi^*(S^*(\g\dual)\otimes \Lambda^* T\dual\! M)_p).\]
This is a sheaf of $C^\infty(G^p\times M)$-modules, hence fine. The morphism between the simplicial levels $\tilde\del_i:\del_i^{-1} \mathcal{C}^p\to \mathcal{C}^{p+1}$ and $\tilde{\sigma}_i:\sigma_i^{-1} \mathcal{C}^{p}\to \mathcal{C}^{p-1}$ are given by pullback along the simplicial bundle maps.

The map $d+\iota:\mathcal{C}^{\bullet,l}\to \mathcal{C}^{\bullet,l+1}$ increases the second grading and is clearly a map of sheaves, as booth operations are local. The maps $\bar d$ and $\bar \iota$ operate between different simplicial levels: On global sections $\bar d$ is the alternating sum of the maps $\tilde\del_i$, while $\bar\iota$ 
\[\bar{\iota} \from C^k(G,S^l(\g\dual)\otimes \Omega^m(M))\to C^{k-1}(G,S^{l+1}(\g\dual)\otimes \Omega^m(M))\]
is given by the formula $\bar\iota =\sum_{i=0}^{k-1}(-1)^i \bar\iota_i$, where each $\bar\iota_i$ is the map of sheaves
\begin{align*}
\bar{\iota}_i\from \sigma_i^{-1} \mathcal{C}^{k}&\to \mathcal{C}^{k-1}\\
(\bar{\iota}_i f)(g_1,\dots,g_{k-1}|X)&=\ddt f(g_1,\dots,g_i,\exp(tX_i),g_{i+1},\dots,g_{k-1}|X), 
\end{align*}
with $X_i=\Ad(g_{i+1}\dots g_{k-1})X$. 

From the discussion of the maps $d+\iota$, $\bar\iota$ and $\bar d$ in \autoref{sec:Getzler} one obtains
that \[\left(\mathcal C^{\bullet,*},d+\iota,\bar\iota\right)\] is a simplicial sheaf homotopy cochain complex.

$\mathcal C^{\bullet,0}$ is the simplicial sheaf of smooth functions, in which the simplicial sheaf $\underline\Z$ injects. This induces a map of simplicial sheaf homotopy cochain complexes 
\[(\underline\Z,0,0)\to \left(\mathcal C^{\bullet,*},d+\iota,\bar\iota\right),\]where $\underline\Z$ is located in degree zero. With respect to this injection, we define

\[\D_C(n)_{G^\bullet\times M}=\Cone(\underline\Integer\oplus\mathcal \mathcal{C}^{\bullet,\geq n}\to\mathcal{C}^{\bullet,*},(z,\omega)\mapsto \omega-z)[-1].\]
\symindex[d]{\D_C(n)_{G^\bullet\times M}}

\begin{defn}
 Let $G$ be a Lie group acting on a smooth manifold M. The \emph{full $G$-equivariant differential cohomology} \index{full equivariant differential cohomology}\symindex[h]{\hat{\mathbb{H}}^*_G(M)} of $M$ is defined to be the cohomology of simplicial sheaf homotopy cochain complexes $\D_C(n)$:
\[\fatH^n_G(M)\defeq H^n(G^\bullet\times M, \D_C(n)_{G^\bullet\times M}).\]
\end{defn}

\begin{theorem}\label{thm:FullHexagon}
 If $G$ is a compact group, one has the following hexagon
\begin{equation} \label{eq:CartanHexagon}
\begin{tikzcd}
\phantom{.} &\hspace{-1em}\faktor{\Omega^{n-1}_G(M)}{(d+\iota)} \arrow{dr}{a} \arrow{rr}{d+\iota}& & \Omega^n_G(M)_\cl\arrow{dr}{}&\\
H_G^{n-1}(M,\C)\arrow{ru}{}\arrow{dr}{}&&\hspace{-.5em}\fatH_G^n(M,\Z) \hspace{-.5em}\arrow{ru}{R}\arrow[two heads]{dr}{I}&&\hspace{-1em}H_G^n(M,\C)\\
&\hspace{-3em}H_G^{n-1}(M,\C/\Z) \arrow[hook]{ur}{} \arrow{rr}{-\beta}&&H_G^n(M,\Z) \arrow{ur}{}&
\end{tikzcd}
\end{equation}
where the line along top, the one along the bottom and the diagonals are exact.
\end{theorem}
\begin{proof}
 This follows by the same arguments as in the last section, and the fact, that for compact Lie groups, the Getzler resolution contracts to the Cartan complex (see Section \ref{sec:Getzler}).
\end{proof}
\begin{example}
 Let $M={pt}$ be a point, then the hexagon \eqref{eq:CartanHexagon} reduces in even degrees to
\[\begin{scaledcd}{.9}
\phantom{.} &0\arrow{dr}{a}\arrow{rr}{d+\iota}& &\left(S^n(\g\dual)\right)^G\arrow{dr}{}&\\
0\arrow{ru}{}\arrow{dr}{}&&\fatH_G^{2n}(pt,\Z)\arrow{ru}{R}\arrow[two heads]{dr}{I}&&\hspace{-1em}H^{2n}(BG,\C)\\
&H^{n-1}(BG,\C/\Z) \arrow[hook]{ur}{} \arrow{rr}{-\beta}&&H^{2n}(BG,\Z) \arrow{ur}{}&
\end{scaledcd}\]
and in odd degrees to
\[
\begin{scaledcd}{.9}
\phantom{.} &\left(S^n(\g\dual)\right)^G\arrow{dr}{a}\arrow{rr}{d+\iota}& &\phantom{I}\!0\phantom{I}\arrow{dr}{}&\\
H^{2n}(BG,\C)\arrow{ru}{}\arrow{dr}{}&&\fatH_G^{2n+1}\!(pt,\Z) \arrow{ru}{R}\arrow[two heads]{dr}{I}&&0\\
&\hspace{-2em}H^{2n}(BG,\C/\Z)\arrow[hook]{ur}{}\arrow{rr}{-\beta}&&H^{2n+1}(BG,\Z) \arrow{ur}{}&
\end{scaledcd}
\]
Hence 
\[\fatH_G^{n}(pt,\Z)=\begin{cases}
                             H^n(BG,\Z)&\text{ if } n \text{ is even}\\
                             H^{n-1}(BG,\C/\Z)&\text{ if } n \text{ is odd}.
                            \end{cases}\]
\end{example}
The contravariant functor $\fatH_G$ assigning an abelian group to the $G$-manifold $M$ is not homotopy invariant, but 
its deviation from homotopy invariance is measured by the homotopy formula. 
\begin{lemma}Let $i_t\from M \to [0,1] \times M$ be the 
inclusion determined by $t \in [0, 1]$ and let $G$ act trivially on the interval.
Let $\omega\in (S^*(\g\dual)\otimes\Omega^*([0,1]\times M))^{n})^G$
    \[(d_M+\iota)\left(\int_{[0,1]\times M/M} \omega\right)= i_1^*\omega-i_0^*\omega+\int_{[0,1]\times M/M} (d_M+\iota)\omega \]
\end{lemma}
\begin{proof}
    Going to local coordinates (using a partition of unity), this is the derivative of the integral by the lower bound, 
the upper bound and the interior derivative.
\end{proof}
\begin{prop}\label{prop:hominv}
    If $\hat x\in\fatH_G^n([0,1]\times M,\Z)$, then 
\[i_1^*\hat x-i_0^*\hat x=a\left(\int_{[0,1]\times M/M} R(\hat x)\right),\]
where we have kept the notions of the previous lemma.
\end{prop}
\begin{proof}
    As equivariant integral cohomology is homotopy invariant, there is a class $y\in H^n(M,\Z)$, such that 
$p_M^*y=I(\hat x)$. As $I$ is surjective, choose a lift $\hat y\in \fatH_G^n(\times M;\Z)$ with $I(\hat y)=y$. Thus 
$I(p_M^*\hat y-\hat x)= 0$ and hence $\hat x=p_M^*\hat y+a(\omega)$ for some 
$\omega\in(S^*(\g\dual)\otimes\Omega^*([0,1]\times M))^{n-1})^G$. Therefore $(d+\iota)\omega=R(a(\omega))=R(\hat x)-R(p_M^*\hat y)$. We 
can write $\omega=dt\wedge \alpha+\beta$, where $dt$ corresponds to the interval and $\alpha,\beta$ are forms on $p_M^* TM$
On the one hand
\[i_1^*\hat x-i_0^*\hat x=a\left(i_1^*\omega-i_0^*\omega\right)=a\left(i_1^*\beta-i_0^*\beta\right).\]
On the other hand
\begin{align*}
a\left(\int_{[0,1]\times M/M} R(\hat x)\right)
&=a\left(\int_{[0,1]\times M/M} R(\hat x)-p_M^*R(\hat y)\right),\\ \intertext{and, as fiber integrals over basic forms vanish,}
&=a\left(\int_{[0,1]\times M/M} (d+\iota)\omega\right)\\
&=a\left(\int_{[0,1]\times M/M} (d_M+\iota)\omega\right)+a\left(\int_{[0,1]\times M/M} d_{[0,1]}\omega\right)\\
&=a\left(\int_{[0,1]\times M/M} (d_M+\iota)dt\wedge \alpha\right)+a\left(\int_{[0,1]\times M/M} d_{[0,1]}\beta\right)\\
&=a\left((i_0^*-i_1^*)dt\wedge \alpha+(d_M+\iota)\left(\int_{[0,1]\times M/M} dt\wedge 
\alpha\right)\right)\\&\hspace{3em}+a\left((i_1^*-i_0^*)\beta\right)\\
&=a\left(i_1^*\beta-i_0^*\beta\right).
\end{align*}
In the last step we use that $a$ vanishes on exact forms.
\end{proof}
To compare our definition with the construction in the last section, we define a subsheaf $\F^1_n\mathcal C^{\bullet,*}\subset \mathcal C^{\bullet,*}$. In the bundle $S^*(\g\dual)\otimes \Lambda^* (T\dual\! M)$, we have the subbundle \[S^{\geq 1}(\g\dual)\otimes \Lambda^{<n} (T\dual\! M)+\left(S^*(\g\dual)\otimes \Lambda^* (T\dual\! M)\right)^{\geq n}.\] $\F^1_n\mathcal C^{\bullet,*}$ is defined to be the sheaf of sections of (the pullback to the simplicial manifold of) this bundle. As one checks immediately
\[\F^1_n\mathcal C^{0,n-1}(M)=\left(\bigoplus_{k=1}^{n/2} S^k(\g\dual)\otimes \Omega^{n-1-2k}(M)\right),\] i.e., the space, whose $G$-invariant part is known from Proposition \ref{pro:f1niso}.
\begin{lemma}
    The image of $\F^1_n\Omega^{\bullet,*}$ under the Getzler map $\mathcal J\from \Omega^{\bullet,*}\to \mathcal C^{\bullet,*}$, defined in section \ref{sec:quasiiso}, lies in $\F^1_n\mathcal C^{\bullet,*}$.
\end{lemma}
\begin{proof}
   Let $U\subset G^p\times M$ an open set and $\omega\in \F^1_n\Omega^{p,k}(U)$. If $k\geq n$ there is nothing to show. Let $k<n$. The projection of the image of $\mathcal J (\omega)$ to $\mathcal C^{\bullet,*}(U)/\F^1_n\mathcal C^{\bullet,*}(U)$ is the part of $\mathcal J (\omega)$ whose polynomial degree is zero. This is zero, since the form degree of $\omega$ on the $G$ part is positive (by the condition $k<n$) and hence $\omega$ is mapped to zero in the quotient and hence to a positive degree polynomial.
\end{proof}
Let $\D_C(1,n)=\Cone(\underline\Integer\oplus\mathcal F^1_n\mathcal C^{\bullet,*}\to\mathcal C^{\bullet,*},(z,\omega,\eta)\mapsto \omega+\eta-z)[-1]$
\begin{lemma} \label{lem:isoGomiNew} The map of chain complexes of simplicial sheaves
\[\mathcal J_*\from\D_{Gomi}(n)_{G^\bullet\times M}\to \D_C(1,n)_{G^\bullet\times M}\]
induces an isomorphism $\hat H_G^*(M,\Z)\to H^*(G^\bullet\times M,\D_C(1,n))$
\end{lemma}
\begin{proof}
    The same arguments as given above show, that $H^*(G^\bullet\times M,\D_C(1,n))$ sits in the same hexagon \eqref{LieHexagon2}, as $\hat H_G^*(M,\Z)$ and the induced maps on all corners is the identity.
\end{proof}
We have an inclusion $\D_C(n)\to\D_C(1,n)$, which, composted with the isomorphism of Lemma \ref{lem:isoGomiNew}, induces a map
\[f\from\fatH_G^*(M,\Z)\to\hat H_G^*(M,\Z).\]
\begin{theorem}
    $f$ is an isomorphism in degree 0,1 and 2 and surjective in higher degrees.
\end{theorem}
\begin{proof}
    This again follows from the hexagons, which coincide in degree $0,1,2$. In higher degrees, the sequence along the bottom is the same and along the top one has surjections.
\end{proof}
\begin{remark}\label{rm:OtherDefs}
    Michael Luis Ortiz discusses an idea of a definition of equivariant differential cohomology in \cite[p.7-9]{Ortiz}. He gives a recipe what to do for general Lie groups, but does not make things precise. In particular he talks about differential forms on $M\times_G EG$. As you will have noted, giving them a precise meaning, in which one can compare them with integral cohomology and the Cartan model is one of the major lines in this thesis and found its final answer in this section.

	On the other hand, there is a definition of Deligne cohomology for orbifolds by Ernesto Lupercio and Bernardo Uribe in \cite{Uribe}. This includes the `action orbifold' of $G$ on $M$ with objects $M$ and morphisms $G\times M$, whose nerve is our simplicial manifold $G^\bullet\times M$. Translating his definition to our language, one gets the complex
	\[\Cone\left(\underline\Integer\oplus \Gamma\left(\cdot,(\del_1^*)^\bullet \Lambda^{\geq n}T\dual\! M\right)^*\to\Gamma\left(\cdot,(\del_1^*)^\bullet \Lambda^*T\dual\! M\right),(z,\omega)\mapsto \omega-z\right)[-1],\]
	of cochain complexes of simplicial sheaves on $G^\bullet\times M$, where $\Gamma(\cdot,E)$ denotes the sheaf of local sections of the bundle $E$. This yields (for $G$ compact) the hexagon:
\[ 
\begin{scaledcd}{.8}
\phantom{.} &\hspace{-5em}\faktor{(\Omega^{n-1}(M))^G}{d\left(\Omega^{n-2}(M)^G\right)} \arrow{dr}{a} \arrow{rr}{d}& & \Omega_\cl^{n}(M,\C)^G\arrow{dr}{}&\\
\faktor{(\Omega^{n-1}_\cl(M))^G}{d\left(\Omega^{n-2}(M)^G\right)}\hspace{-2em} \arrow{ru}{}\arrow{dr}{}&&\hspace{-1em}\hat H_G^n(M,\Z) \hspace{-1em}\arrow{ru}{R}\arrow[two heads]{dr}{I}&&\hspace{-1em}\faktor{(\Omega^{n}_\cl(M))^G}{d\left(\Omega^{n-1}(M)^G\right)} \\
&\hspace{-3em}H_G^{n-1}(M,\Cone(\underline\Z\to \Omega^*(\cdot)^G)) \arrow[hook]{ur}{} \arrow{rr}{-\beta}&&H_G^n(M,\Z) \arrow{ur}{}&
\end{scaledcd}
\]
In the case of finite groups, one has $H_G^*(M,\C)=\Omega^{n}_\cl(M)^G/d\Omega^{n-1}(M)^G$, thus this is the same as we had before. In the case of positive dimensional Lie groups it is even less satisfactory then the definition of Gomi, as there is not even equivariant complex cohomology at the left and the right end.
\end{remark}

\section{Equivariant differential characteristic classes}\label{sec:eqdiffchar}

\subsection{Definitions}

Let us restrict to compact groups $G$ acting on the manifold and on vector bundles. As rank $n$ vector bundles admit a hermitian metric, they are in one to one correspondence with principal $U(n)$-bundles. Thus any characteristic form for vector bundles corresponds to an invariant polynomial $P\in I^*(U(n))$ (see, e.g. \cite[Corollary 5.13]{Ich1}. 

Let $E\to M$ be a $G$-equivariant vector bundle. Recall that a connection is map
\[\nabla\from\Omega^0(M,E)\to \Omega^1(M,E),\] which satisfies a Leibniz rule\[\nabla(f\varphi)=df\wedge \varphi+f\nabla \varphi \text{ for } f\in\Omega^0(M,\C),\varphi\in\Omega^0(M,E).\] Further, a connection $\nabla$\symindex[n]{\nabla} extends uniquely to a $\C$-linear map \[\nabla\from \Omega^*(M,E)\to \Omega^{*+1}(M,E),\] called exterior connection, by imposing the sign respecting Leibniz rule 
\[\nabla(\omega\wedge \varphi)=d\omega\wedge \varphi+(-1)^k \omega\wedge \nabla \varphi \text{ for } \omega\in\Omega^k(M,\C),\varphi\in\Omega^*(M,E).\]

One observes, that $\nabla\circ\nabla\from\Omega^0(M,E)\to \Omega^2(M,E)$ is $C^\infty$-linear and hence given by left multiplication with an endomorphism valued 2-form, which is known as the curvature operator $R^\nabla\in  \Omega^2(M,\End E)$. If the connection is $G$-invariant, then there is another associated map.
\begin{defn}[Def.\ 2.23.\ of \cite{Bunke}]\label{def:momentmap}
    Let $\nabla$ be a $G$-invariant connection on the $G$-vector bundle $\mathcal E$. The \emphind{moment map} $\mu^\nabla\in\Hom(\mathfrak{g}, \omega^0(M, \End(\mathcal E)))^G$ \symindex[m]{\mu^\nabla} is defined by
\[\mu^\nabla(X) \wedge \varphi \defeq \nabla_{X^\sharp_M} \varphi + L^{\mathcal E}_X \varphi,\quad \varphi \in \omega^0(M, \mathcal E).\]
Here $L^{\mathcal E}_X$ \symindex[l]{L^M_X} denotes the derivative
\[L^{\mathcal E}_X \varphi=\ddt\exp(tX)^*\varphi.\]
\end{defn}

From any invariant polynomial we obtain a equivariant differential forms of the $G$-invariant connection by \[\omega(\nabla)=P(R^\nabla+\mu^\nabla)\in\Omega_G(M).\]

Moreover, if $\omega$ is integral, i.e. has integral periods, then there is an integral equivariant characteristic class $c^\omega$ coinciding with the class of $\omega$ in complex cohomology.
\begin{defn}
    A differential refinement of $\omega$ associates to every $G$-equivariant vector bundle with connection $(E,\nabla)$ on 
$M$ a class $\hat{\omega}(\nabla)\in \fatH_G(M;\Z)$ such that
\[R(\hat{\omega}(\nabla))= \omega(\nabla),\quad I(\omega(\nabla))=c^\omega(E)\]
and for every map $f\from M\to M'$, we have $f^*\hat{\omega}(\nabla)=\hat{\omega}(f^*\nabla)$.
\end{defn}
As the intersection of the kernels
\[\ker(R)\cap\ker(I)=\faktor{H^{n-1}_G(M,\C)}{H^{n-1}_G(M,\Z)}\]
is in general non-trivial, the differentially refined class $\hat{\omega}(\nabla)$ can contain finer information than the 
pair $(\omega(\nabla),c^\omega(E))$. Thus it is a priori not clear that for a given equivariant characteristic form, there is only one equivariant differential characteristic class. 

\begin{theorem}
    An integral equivariant characteristic form admits a unique equivariant differential extension.
\end{theorem}
The line of arguments proof this assertion is (almost) the following: A simplicial manifold model of the universal $U(n)$-bundle is given by (compare \cite[Section 5]{Dupont}) the simplicial principal $U(n)$-bundle $\gamma\from N\overline{U(n)}_\bullet\to NU(n)_\bullet$, with \[N\overline{U(n)}_p=U(n)^{p+1},\] $\del_i$ removes the $i$-th coefficient, and $\sigma_i$ doubles the $i$-th coefficient. $NU(n)=U(n)^\bullet\times pt$ and $\gamma(g_0,\dots,g_p)=(g_0g_1\invers,\dots,g_{p-1}g_p\invers)$. As $\fatH_{U(n)}^{2n}(pt,\Z)=H^n(BU(n),\Z)$, we would like to define a map of simplicial manifolds $G^\bullet\times M\to NU(n)$ classifying our bundle and pull back the universal class together with a corresponding connection. Now we can compare this connection with the one defined on our bundle and change the differential characteristic class according to this.

\begin{lemma}\label{lem:homotopy}
    Let $\nabla$ and $\nabla'$ are two connections on the same bundle, then
\[\hat{\omega}(\nabla)-\hat{\omega}(\nabla')=a(\tilde{\omega}(\nabla,\nabla'))\]
\end{lemma}
\begin{proof}
Let $\nabla_t$ denote the convex combination of $\nabla$ and $\nabla'$. Then by Proposition \ref{prop:hominv}
\begin{align*}
    \hat{\omega}(\nabla)-\hat{\omega}(\nabla')&=i_1^*\hat{\omega}(\nabla_t)-i_0^*\hat{\omega}(\nabla_t)\\
&=a\left(\int_{[0,1]\times M/M} R(\hat{\omega}(\nabla_t))\right)\\
&=a\left(\int_{[0,1]\times M/M} \omega(\nabla_t)\right)\\
&=a(\tilde{\omega}(\nabla,\nabla')).
\end{align*}
\end{proof}
This Lemma implies, in particular, that we are done, if we have defined the refined for hermitian bundles with hermitian connection, since any connection can by symmetrized (compare \cite[Section 2.5]{Bunke}).

To construct the classifying map we will need an intermediate bundle, for which one can easily write down pullback maps to the given bundle and to the universal bundle. Therefor we need to recall the following construction from \cite[Section 4]{Ich1}.

Let $\mathcal U=\{U_\alpha|\alpha\in A\}$ be an open cover of some $G$-manifold $M$. This induces a simplicial cover of $G^\bullet \times M$: 
Define the simplicial index set $A^{(p)}=A^{p+1}$ with face and degeneracy maps given by removing respective doubling of the $i$-th element. Then define the simplicial cover $U^{(p)}=\{U^{(p)}_\alpha\}_{\alpha\in A^{(p)}}$  inductively by
\[U^{(p)}_\alpha=\bigcap_{i=0}^{p}\del_i^{-1}\left(U^{(p-1)}_{\del_i(\alpha)}\right),\]
where $U^{(0)}_\alpha=U_\alpha$ for any $\alpha\in A^{(0)}=A$. 

From this simplicial cover one obtains the the simplicial manifold $(G^\bullet\times M)_\mathcal{U}$ as
\[\left((G^\bullet\times M)_\mathcal{U}\right)_p\defeq\coprod_{(\alpha_0,\dots,\alpha_p)} U^{(p)}_{\alpha_0}\cap\dots\cap U^{(p)}_{\alpha_p},\]
where the disjoint union is taken over all $(p+1)$-tuples $(\alpha_0,\dots,\alpha_p)\in (A^{(p)})^{p+1}$ with $U^{(p)}_{\alpha_0}\cap\dots\cap U^{(p)}_{\alpha_p}\neq \emptyset$. The face and degeneracy maps are given on the index sets $(A^{(p)})^{p+1}$ by removing, respective doubling of the $i$-th index and on the open sets by the corresponding inclusions composed with the $i$-th face and degeneracy map of $G^\bullet \times M$.

Let $\pi\from E\to M$ be a $G$-equivariant hermitian vector bundle with hermitian connection $\nabla$ and $B$ be the associated 
principal $U(n)$-bundle furnished with the associated principal connection $\vartheta$. From an open cover $\mathcal U$ of $M$ we obtain the cover $\pi\invers\mathcal U$ of $B$ and thus the construction above yields a simplicial bundle
\[\pi\colon (G^\bullet\times B)_{\pi^{-1}\mathcal U}\to (G^\bullet\times M)_{\mathcal U}.\]
and the commutative diagram
\[\begin{tikzcd}
 (G^\bullet\times B)_{\pi\invers\mathcal U}\arrow{d}{\pi_{\mathcal U}}\arrow{r}&G^\bullet\times B\arrow{d}\\
(G^\bullet\times M)_{\mathcal U}\arrow{r}&G^\bullet\times M
\end{tikzcd}\]
induced by the inclusions of the covering sets is a pullback, since the cover we take on $G^\bullet\times B$ is induced by $\pi$ and $\mathcal U^\bullet$.

Suppose the cover $\mathcal U=\{U_\alpha\}_{\alpha\in A}$ of $M$ trivializes $B$ with trivialization \[\varphi_\alpha\from V_\alpha=\pi^{-1}(U_\alpha)\to U_\alpha\times U(n)\] and transition functions $g_{\alpha\beta}\colon U_\alpha\cap U_\beta\to U(n)$. Then there is an induced map \[\overline{\psi}\from (G^\bullet\times B)_{\pi^{-1}\mathcal U}\to N\overline U(n),\] which is given on the intersection of $p+1$ covering sets of $G^p\times B$ 
\[V=\bigcap_{j=0}^p V^{(p)}_{\alpha^j_0,\dots,\alpha^j_p}\]
by 
\[(g_1,\dots,g_p,x)\mapsto (\varphi_{\alpha^0_0}(g_1\dots g_px),\varphi_{\alpha^1_1}(g_2\dots g_px),\dots,\varphi_{\alpha^p_p}(x))\in U(n)^{p+1},\]
where, on the right-hand side, the maps $\varphi_\alpha$ are understood to be composed with the projection to $U(n)$.

Next, we want to define $\psi:(G^\bullet\times M)_{\mathcal U}\to NU(n)$, such that $\overline\psi$ covers $\psi$. Therefore we need some additional transition functions of the bundle. Define 
\begin{align*}
h_{\alpha\beta}\colon G\times M\supset\del_0\invers U_\alpha\cap\del_1\invers U_\beta&\to U(n)\\
(g,m)&\mapsto (\pi_2\circ \varphi_\alpha(gx))(\pi_2\circ \varphi_\beta(x))\invers,
\end{align*}
for any $x\in \pi\invers(m)$. Define $\psi$ on \[U=\bigcap_{j=0}^q U^{(p)}_{\alpha^j_0,\dots,\alpha^j_p}\] by 
\begin{multline}(g_1,\dots,g_p,m)\mapsto (h_{\alpha^0_0\alpha^1_1} (g_1,g_2\dots g_pm),\\h_{\alpha^1_1\alpha^2_2} (g_2,g_3\dots g_pm),\dots,h_{\alpha^{p-1}_{p-1}\alpha^p_p} (g_p,m),*).\end{multline}

These maps combine to the following commutative diagram of simplicial manifolds:
\[\begin{tikzcd}
    G^\bullet\times B\arrow{d}{\pi}&(G^\bullet\times B)_{\pi\invers\mathcal U}\arrow{l}[above]{\bar{i}}\arrow{r}{\bar{\varphi}}\arrow{d}{}&N\bar 
U(n)\arrow{d}\\
    G^\bullet\times M&(G^\bullet\times M)_{\mathcal U}\arrow{l}[above]{i}\arrow{r}{\varphi}&NU(n).
\end{tikzcd}\]

\begin{prop}
The map $i$ induces an isomorphism
\[i^*\from\fatH_G^n(M,\Z)\to H^n((G^\bullet\times M)_{\mathcal U},i^*\D_C(n))\]
and isomorphisms between all corners of the hexagons \eqref{eq:CartanHexagon} with the corresponding corners of
\[\begin{scaledcd}{.74}
\phantom{.} &\hspace{-4em}H^n((G^\bullet\times M)_{\mathcal U},i^*\Cone(\mathcal{C}^{\bullet,\geq n}\to\mathcal{C}^{\bullet,*})[-1]) \arrow{dr}{a} \arrow{rr}{d+\iota}& & H^n((G^\bullet\times M)_{\mathcal U},i^*\mathcal{C}^{\bullet,\geq n}) \arrow{dr}{}&\\
H_G^{n-1}(M,\C)\arrow{ru}{}\arrow{dr}{}&&\hspace{-1.5em}H^n((G^\bullet\times M)_{\mathcal U},i^*\D_C(n))\hspace{-1.5em}\arrow{ru}{R}\arrow[two heads]{dr}{I}&&\hspace{-1em}H_G^n(M,\C).\\
&\hspace{-3em}H_G^{n-1}(M,\C/\Z) \arrow[hook]{ur}{} \arrow{rr}{-\beta}&&H_G^n(M,\Z) \arrow{ur}{}&
\end{scaledcd}\]
\end{prop}
\begin{proof}
Recall that $\|i\|\from\|(G^\bullet\times M)_{\mathcal U}\|\to \|G^\bullet\times M\|$ is a homotopy equivalence.
 The short exact sequence of simplicial sheaves 
\[0\to\Cone(\underline\Integer\to\mathcal{C}^{\bullet,*})\to \D_C(n)_{G^\bullet\times M}\to \mathcal 
\mathcal{C}^{\bullet,\geq n}\to 0\]
and the map $i$ induce the following diagram with exact rows
\[\begin{scaledcd}{.8}
0\arrow{r}\arrow{d}{=}&H^{n-1}_G(M,\C/\Z)\arrow{r}\arrow{d}{=}&\fatH_G^n(M,\Z)\arrow{r}\arrow{d}{i^*}
&\Omega_G^n(M)_\cl  \arrow{r}\arrow{d}{i^*}&H^{n-1}_G(M,\C/\Z)\arrow{d}{=}\\
    0\arrow{r}&H^{n-1}_G(M,\C/\Z)\arrow{r}&H^n((G^\bullet\times M)_{\mathcal U},i^*\D_C(n))\arrow{r}&H^n((G^\bullet\times M)_{\mathcal 
U},i^*\mathcal{C}^{*,\geq n})  \arrow{r}&H^{n-1}_G(M,\C/\Z).
\end{scaledcd}\]
Thus, by the five lemma, it is sufficient to show, that 
\[i^*\from \Omega_G^n(M)_\cl\to H^n((N_GM)_{\mathcal U},i^*\mathcal{C}^{\bullet,\geq n})\]
is an isomorphism. Observe that
\begin{multline*}H^n((N_GM)_{\mathcal U},i^*\mathcal{C}^{\bullet,\geq n})=\ker\left(d+\iota\from 
\mathcal{C}^{0,n}\left(\coprod U_\alpha\right)\to \mathcal{C}^{0,n+1}\left(\coprod 
U_\alpha\right)\right)\cap\\\ker\left(\del\from \mathcal{C}^{0,n}\left(\coprod U_\alpha\right)\to 
\mathcal{C}^{1,n}\left(\coprod_{\alpha_1,\alpha_2,\beta_1,\beta_2} U^{(1)}_{\alpha_1\alpha_2}\cap U^{(1)}_{\beta_1\beta_2}\right)\right).
\end{multline*}
Let $(\omega_\alpha)\in \mathcal{C}^{0,n}\left(\coprod U_\alpha\right)$. The definition of the map $\del$ 
\[\begin{tikzcd}[row sep=tiny,column sep=large]
\phantom{.}&\phantom{g}m\in U_{\beta_2}\\
U^{(1)}_{\alpha_1\alpha_2}\cap U^{(1)}_{\beta_1\beta_2}\ni (g,m)\arrow[mapsto]{ur}{\del_0}    \arrow[mapsto,swap]{dr}{\del_1}&\\
&gm\in U_{\alpha_1}
\end{tikzcd}\]
implies that $\del (\omega_\alpha)=0$ is equivalent to
\[\del^*_0\omega_\beta|_{U^{(1)}_{\alpha\beta}}=\del_1^*\omega_\alpha|_{U^{(1)}_{\alpha\beta}}.\]
Moreover, since ${e}\times (U_\alpha\cap U_\beta)\subset \del_1^*U_\alpha\cap \del_0^*U_\beta=U^{(1)}_{\alpha\beta}$, this equation implies that 
$(\omega_\alpha)$ is the restriction of a global section $\omega\in C^{0,n}(M)$, which is by the same equation $G$-invariant. Hence $\omega\in \ker(d+\iota)=\Omega_G^n(M)_\cl$. This proves the 
first claim.

The claim about the hexagon follows by the same argument, because the `de Rham' sequence along the top is exact.
\end{proof}

One defines (compare \cite[p.94]{Dupont}) a connection $\bar{\vartheta}$ on $N\overline{U(n)}\to NU(n)$: Let $\vartheta_0\in \Omega^1(K,\mathfrak k)$ denote the unique connection of the trivial bundle $K\to \mathrm{pt}$, i.e., \[\vartheta_0(k)=L_{k\invers}\from T_kK\to T_eK=\mathfrak k.\] 
Let \[\pi_{i}\from \Delta^p\times K^{p+1}\to K\] denote the projection to the $i$-th coefficient, $i=0,\dots,p$ and $\vartheta_{i}=\pi_{i}^*\vartheta_0$. Then we define $\bar{\vartheta}$ on $\Delta^p\times (N\overline K)_{p}$ by \[\bar{\vartheta}=\sum_{i} t_i\vartheta_{i},\] where $(t_0,\dots,t_p)$ are barycentric coordinates on the simplex. $\bar{\vartheta}|_{\Delta^p\times (N\overline K)_p}$ is a connection on $\Delta^p\times (N\overline K)_p$, as it is a convex combination of connections. It can be seen easily from the definition, that $\bar{\vartheta}$ is a  simplicial Dupont 1-form. For more details see also \cite{Ich1}.

Let $P\in I^*(U(n))$ denote the polynomial and $c_P\in H^n(BU(n),\Z)=\fatH_{U(n)}^{2n}(pt,\Z)$ the universal characteristic class corresponding to the integral characteristic form $\omega_P$.
\begin{defn2}
    The differential refinement is given by the formula
\[\hat{\omega}(\nabla)=(i^*)\invers(\varphi^*c_P+a(\widetilde{\omega_P}(i^*\vartheta,\bar{\varphi}^*\vartheta_0))).\]
This definition is independent of the chosen cover and trivializations and defines the differential refinement of the integral characteristic form $\omega_P$. 
\end{defn2}
\begin{proof}
We will prove the independence of the cover in three steps:

\emph{Step 1:} Let $\mathcal U'=\{U'_\beta\}$ be a refinement of the cover $\mathcal U$, i.e., for any $\beta$, there is some $\alpha(\beta)$, such that $U'_\beta\subset U_{\alpha(\beta)}$; let $\varphi_\beta'=\varphi_{\alpha(\beta)}|_{U'_\beta}$. The inclusion of the refinement yields a commutative diagram
\[\begin{tikzcd}[row sep=scriptsize, column sep=scriptsize]
G^\bullet\times B\arrow{dd}\arrow[equal]{dr}&&(G^\bullet\times B)_{\pi\invers \mathcal U'}\arrow{ll}[above]{\bar i'}\arrow{rr}{\bar{\varphi}'}\arrow{dd}{}\arrow{dr}&&N\bar U(n)\arrow{dd}\arrow[equal]{dr}&\\    
&G^\bullet\times B&&(G^\bullet\times B)_{\pi\invers\mathcal U}\arrow[crossing over]{ll}[above, near end]{\bar i}\arrow[crossing over]{rr}[near start]{\bar{\varphi}}&&N\bar U(n)\arrow{dd}\\
G^\bullet\times M\arrow[equal]{dr}&&(G^\bullet\times M)_{\mathcal U'}\arrow{ll}[above, near start]{i'}\arrow{rr}[near start]{\varphi'}\arrow{dr}&&NU(n)\arrow[equal]{dr}&\\
& G^\bullet\times M\arrow[crossing over,leftarrow]{uu}{}&&(G^\bullet\times M)_{\mathcal U}\arrow{ll}{i}\arrow{rr}[below]{\varphi}\arrow[crossing over,leftarrow]{uu}{}&&NU(n)
\end{tikzcd}\]
from which the independence of the cover follows, because the direct pullback is same as the one factorized over the coarser cover.

\emph{Step 2:} Take one cover $\mathcal U=\{U_\alpha\}$, with two different families of trivialization maps $\varphi_\alpha,\varphi'_\alpha\from \pi\invers U_\alpha\to U_\alpha\times G$. 

Then there is a family of maps $\psi_\alpha\from U_\alpha\to G$, such that $\psi_\alpha(\pi(b))\cdot \varphi_\alpha(b)=\varphi_\alpha'(b)$ for any $b\in \pi\invers U_\alpha$ and any $\alpha$. 

The difference of the two definitions is
\begin{multline*}
\varphi^*c_P+a(\widetilde{\omega_P}(i^*\vartheta,\bar{\varphi}^*\vartheta_0))-\varphi'^*c_P-a(\widetilde{\omega_P}(i^*\vartheta,\bar{\varphi'}^*\vartheta_0))\\=\varphi^*c_P-\varphi'^*c_P-a(\widetilde{\omega_P}(\bar{\varphi}^*\vartheta_0,\bar{\varphi'}^*\vartheta_0))    
\end{multline*}
First assume each $U_\alpha$ is contractible, then there is a homotopy $\widetilde{\psi}_\alpha\from [0,1]\times U_\alpha\to G$ such that $i_1^*\widetilde{\psi}_\alpha=\psi_\alpha$ and $i_0^*\widetilde{\psi}_\alpha$ maps any point to $e\in G$. These homotopies induce a homotopy 
\[\widetilde{\varphi}\from [0,1]\times (G^\bullet\times B)_{\pi\invers \mathcal U'}\to NU(n)\]
 between $\tilde{\varphi}_0=\varphi$ and $\tilde{\varphi}_1=\varphi'$ and one can calculate
\begin{align*}
\varphi^*c_P-\varphi'^*c_P&=i_0^*\tilde{\varphi}^*c_P-i_1^*\tilde{\varphi}^*c_P\\
&=a\left(\int_{[0,1]} R(\tilde{\varphi}^*c_P)\right)\\
&=a\left(\int_{[0,1]} \tilde{\varphi}^*R(c_P)\right)\\
&=a\left(\int_{[0,1]} \tilde{\varphi}^*\int_\Delta P(\vartheta_0)\right)\\
&=a\left(\int_{[0,1]} \int_\Delta P(\tilde{\varphi}^*\vartheta_0)\right)\\
&=a(\tilde{\omega_P}(\bar{\varphi}^*\vartheta_0,\bar{\varphi'}^*\vartheta_0)).
\end{align*}
In the last step, we use that $\widetilde{\omega}_P$ is independent of the path between the connections.

The case of non contractible $U_\alpha$ follows by Step 1.

\emph{Step 3:} Let $(\mathcal U,(\varphi_\alpha)),(\mathcal U',(\varphi'_\beta))$ be two different covers with trivializations. Let $\tilde{\mathcal U}=\{U_\alpha\cap U'_\beta|\alpha,\beta\}$ be the common refinement on which there are two different families of trivializations are introduced by $\varphi$ and $\varphi'$. Now the statement follows from the previous steps.

Next, we check the properties of the differential refinement: \[I(\hat{\omega}(\nabla))=I((i^*)\invers(\|\varphi\|^*c_P))=c^\omega(B)\] and
\begin{align*}
	R(\hat{\omega}(\nabla))&=R((i^*)\invers(\|\varphi\|^*c_P)+a(\tilde{\omega}(i^*\vartheta,\bar{\varphi}^*\vartheta_0))\\
	&=R((i^*)\invers(\|\varphi\|^*c_P))+(d+\iota)\tilde{\omega}(i^*\vartheta,\bar{\varphi}^*\vartheta_0)\\
	&=(i^*)\invers(\omega(\bar{\varphi}^*\vartheta_0)+\omega(i^*\vartheta)-\omega(\bar{\varphi}^*\vartheta_0))\\
	&=\omega(\nabla).   
\end{align*}

Let $(F,f)\from (B,M)\to (B',M')$ be a pullback. As a trivialization of $(B',M')$ induces a trivialization of $(B,M)$, one has a commutative diagram
\[\begin{tikzcd}[row sep=scriptsize, column sep=tiny]
G^\bullet\times B\arrow{dd}\arrow{dr}&&(G^\bullet\times B)_{\pi\invers f\invers\mathcal U}\arrow{ll}\arrow{rr}\arrow{dd}{}\arrow{dr}&&N\bar U(n)\arrow{dd}\arrow[equal]{dr}&\\    
&G^\bullet\times B'&&(G^\bullet\times B')_{\pi\invers\mathcal U}\arrow[crossing over]{ll}{}\arrow[crossing over]{rr}{}&&N\bar U(n)\arrow{dd}\\
G^\bullet\times M\arrow{dr}&&(G^\bullet\times M)_{f\invers\mathcal U}\arrow{ll}\arrow{rr}\arrow{dr}&&NU(n)\arrow[equal]{dr}&\\
& G^\bullet\times M'\arrow[crossing over,leftarrow]{uu}{}&&(G^\bullet\times M')_{\mathcal U}\arrow{ll}\arrow{rr}\arrow[crossing over,leftarrow]{uu}{}&&NU(n)
\end{tikzcd}\]
which clearly implies the pullback property.

The refinement is unique, since we used for our definition only properties the differential refinement necessarily has, namely the pullback-property and Lemma \ref{lem:homotopy}.
\end{proof}

\subsection{Multiplicative structures}
\begin{defn}[compare {\cite[Definition 3.94]{Bunke}}] Let $G$ be a compact Lie group. A product on equivariant Deligne cohomology is the datum of a graded commutative ring structure (denoted by $\cup$) on $\fatH_G^*(M,\Z)$ for every $G$-manifold M such that
\begin{enumerate}    
    \item $f^*\from \fatH_G^*(M,\Z)\to \fatH_G^*(M',\Z)$is a homomorphism of rings for every smooth map $f\from M'\to M$,
    \item $R\from \fatH_G^*(M,\Z)\to \Omega^*_G(M)_\cl$ is multiplicative for all $M$,
    \item $I\from \fatH_G^*(M,\Z)\to H_G^*(M,\Z)$ is multiplicative for all $M$, and
    \item $a(\alpha) \cup x = a(\alpha \wedge R(x))$ for all $\alpha\in \omega^*_G(M;\C)/\im(d+\iota)$ and $x\in\fatH_G^*(M,\Z)$
\end{enumerate}
\end{defn}
\begin{prop}
    There exists a unique product on equivariant Deligne cohomology.
\end{prop}
\begin{proof}
    \emph{Uniqueness} follows by (almost) verbatim the same arguments as given by \cite[p. 60]{Bunke}: The difference between two products
    \[B=\cup'-\cup\from \fatH_G^p(M,\Z)\otimes\fatH_G^q(M,\Z)\to \fatH_G^{p+q}(M,\Z)\] factorizes over a bilinear map
    \[\tilde B\from H_G^p(M,\Z)\otimes H_G^q(M,\Z)\to H_G^{p+q}(M,\C/\Z),\]
    by the hexagon \eqref{eq:CartanHexagon}, since $R\circ B=0$ and $B\circ (a\times \id)=0$. The bilinear map $\tilde B$ corresponds to a map of Eilenberg-Maclane spaces
    \[K(Z, p) \wedge K(Z, q) \to  K(C/Z, p + q - 1),\] which is homotop to a constant map, as the smash product on the left hand side is $p+q-1$-connected.

 \emph{Existence:} We will leave this to the reader. The idea is to copy the arguments of \cite[Section 3.4]{Bunke}, but replace the de Rham $d$ in the definition of the map on the level of chain complexes \cite[Equation (29)]{Bunke} by the boundary map $\bar d+\bar{\iota}+ d+ \iota$ of Getzler.
 \end{proof}

Recall that the total equivariant differential Chern class is the sum of the equivariant differential Chern classes
\[\hat c(\nabla)=1+\hat c_1(\nabla)+\hat c_2(\nabla)+\dots\in \bigoplus_{n\text{ even}}\fatH_G^n(M,\Z).\]

\begin{prop}
The total equivariant differential Chern class satisfies a Whitney sum formula, i.e., given two $G$-equivariant vector bundles $(E,\nabla),(E',\nabla)$ with equivariant connection over the $G$-manifold $M$ and let $\nabla\oplus\nabla'$ be the Whitney sum connection on $E\oplus E'$, then 
\[\hat c(\nabla\oplus\nabla')=\hat c(\nabla)\cup \hat c(\nabla').\]
\end{prop}
\begin{proof}
The proof consists of two steps: First we will prove the formula for the classifying space and afterwards, we will show that the difference terms fit.

Since the $U(n)$-equivariant differential cohomology of a point equals in even dimension the $U(n)$-equivariant integral cohomology of a point, the formular follows from the non-differential Whitney sum formula and the compatibility of the cup products.

Thus by construction of the equivariant differential characteristic classes, we do only have to check, that the difference terms fit, i.e., the classifying maps of $E, E'$ and $E\oplus E'$ induce connections $\nabla_0, \nabla_0'$ and $\nabla_0\oplus\nabla_0'$, for which \[\hat c(\nabla_0\oplus\nabla'_0)=\hat c(\nabla_0)\cup \hat c(\nabla'_0)\] holds by the pullback property of $\hat c$ and the first step for the universal bundles.

Denote the characteristic form of $c$ by $\omega$ and the transgression form by $\widetilde{\omega}$. Now calculate by applying the properties of the cup product
\begin{align*}
\hat c(\nabla)\cup &\hat c(\nabla')\\
&=\left(\hat c(\nabla_0)+a(\widetilde{\omega}(\nabla,\nabla_0))\right)\cup \left(\hat c(\nabla_0')+a(\widetilde{\omega}(\nabla',\nabla_0'))\right)\\
&=\hat c(\nabla_0\oplus\nabla_0')+a(\widetilde{\omega}(\nabla,\nabla_0))\cup\hat c(\nabla_0')) 
\\&\qquad+\hat c(\nabla_0)\cup a(\widetilde{\omega}(\nabla',\nabla_0'))+a(\widetilde{\omega}(\nabla,\nabla_0))\cup a(\widetilde{\omega}(\nabla',\nabla_0'))\\
&=\hat c(\nabla_0\oplus\nabla_0')+a(\widetilde{\omega}(\nabla,\nabla_0)\wedge R(\hat c(\nabla_0'))) 
\\&\qquad+a(R(\hat c(\nabla_0))\wedge(\widetilde{\omega}(\nabla',\nabla_0'))+a(\widetilde{\omega}(\nabla,\nabla_0)\wedge R\circ a(\widetilde{\omega}(\nabla',\nabla_0')))\\
&=\hat c(\nabla_0\oplus\nabla_0')+a(\widetilde{\omega}(\nabla,\nabla_0)\wedge \omega(\nabla_0'))
\\&\qquad +a(\omega(\nabla_0)\wedge(\widetilde{\omega}(\nabla',\nabla_0'))+a(\widetilde{\omega}(\nabla,\nabla_0)\wedge (\omega(\nabla')-\omega(\nabla_0')))\\
&=\hat c(\nabla_0\oplus\nabla_0')+a(\widetilde{\omega}(\nabla,\nabla_0)\wedge \omega(\nabla')) +a(\omega(\nabla_0)\wedge(\widetilde{\omega}(\nabla',\nabla_0'))\\
&=\hat c(\nabla_0\oplus\nabla_0')+a(\widetilde{\omega}(\nabla\oplus\nabla',\nabla_0\oplus\nabla')+a((\widetilde{\omega}(\nabla_0\oplus\nabla',\nabla_0\oplus\nabla_0'))\\
&=\hat c(\nabla\oplus\nabla').
\end{align*}
\end{proof}
 
\section{Examples for equivariant differential cohomology}
\makeatletter{}\subsection{Free actions}
Let the Lie group $G$ act freely on the manifold $M$ from the left. Does equivariant differential cohomology groups make a difference between the $G$ manifolds $M$ and $G\times M/G$? As equivariant cohomology does not make one, the question reduces to differential forms. 

To discuss this, we collect the following statements:
\begin{defn}[Def.\ 13.5.\ of \cite{LieAction}]
    The action is \emphind{proper}, if the action map
\[G \times M \to M \times M, (g, m) \mapsto (g m, m)\]
is proper, i.e., the pre-image of any compact set is compact.
\end{defn}
\begin{theorem}[Th.\ 13.8.\ of \cite{LieAction}]
Suppose $G$ acts properly on $M$. Then each orbit $G\cdot m$ is an embedded closed submanifold of $M$, with
\[T_m(G \cdot m) = \{X^\sharp_M(m) | X \in \mathfrak g\}=\mathfrak{g}^\sharp_m.\]\symindex[g]{\mathfrak{g}^\sharp}
\end{theorem}
\begin{theorem}[Th.\ 13.10.\ of \cite{LieAction}]
Suppose that $G$ acts properly and freely on $M$, then the orbit space $M/G$ is a manifold and the quotient map $\pi \from M \to M/G$ is a submersion.
\end{theorem}

Suppose the action is free and proper, thus $M/G$ is a manifold. The quotient map always induce injections
\[q^*\from\Omega^n(M/G)\to \Omega^n(M)^G\]
and
\[\pr^*\from\Omega^n(M/G)\to \Omega^n(G\times M/G)^G.\]
These lead to two resolutions of $\Omega^*(M/G)$: The first one is given as the double complex
\[\begin{tikzcd}[column sep=small]
\vdots\arrow{d}&\vdots\arrow{d}&\vdots\arrow{d}&\vdots\arrow{d}&\\
  \Omega^n(M/G)\arrow{r}{q^*}\arrow{d}{d} &\Omega^n(M)^G\arrow{r}{\iota}\arrow{d}{d}&\left(\mathfrak g\dual\otimes \Omega^{n-1}(M)\right)^G\arrow{r}{\iota}\arrow{d}{d}&\left(S^2(\mathfrak g\dual)\otimes \Omega^{n-2}(M)\right)^G\arrow{r}\arrow{d}{d}& \dots\\
\Omega^{n+1}(M/G)\arrow{r}{q^*}\arrow{d}{d} &\Omega^{n+1}(M)^G\arrow{r}{\iota}\arrow{d}{d}&\left(\mathfrak g\dual\otimes \Omega^{n}(M)\right)^G\arrow{r}{\iota}\arrow{d}{d}&\left(S^2(\mathfrak g\dual)\otimes \Omega^{n-1}(M)\right)^G\arrow{r}\arrow{d}{d}& \dots  \\
\vdots&  \vdots&\vdots&\vdots&
\end{tikzcd}\]
whose total complex is the Cartan complex $\Omega^*_G(M)$, while the total complex of the second resolution is $\Omega^*_G(G\times M/G)$. The question now is: Are these two complexes equivalent on the level of cycles? This is clearly true for zero forms as the two maps
\[C^\infty(M)^G\overset{q^*}{\longleftarrow}C^\infty(M/G)\overset{\pr^*}{\longrightarrow}C^\infty(G\times M/G)^G\] 
are isomorphisms. For higher degrees let $h$ be a $G$-invariant Riemannian metric on $M$. Then the tangent bundle 
 \[TM=\mathfrak{g}^\sharp\oplus \left(\mathfrak{g}^\sharp\right)^\bot\]
 splits with respect to $h$. Moreover $dq_m\from \left(\mathfrak{g}^\sharp_m\right)^\bot\to T_{q(m)}(M/G)$ is an isomorphism for any $m\in M$. Thus we have the following lemma, what shows the equivalence in degree one 
 \begin{lemma}
     Let $G$ act properly and freely on $M$, then
 \[0\to \Omega^1(M/G)\overset{q^*}{\to}\Omega^1(M)^G\overset{\iota}{\to} \left(\mathfrak g\dual\otimes\Omega^0(M)\right)^G\to 0\]
 splits.
 \end{lemma}
 \begin{proof}
     Restriction to $\left(\mathfrak{g}^\sharp\right)^\bot\subset TM$ defines a map $\Omega^1(M)^G\to\Omega^1(M/G)$ which is left inverse of $q^*$. Thus it is a split.
 \end{proof}
For the higher degrees, recall the following relation between exterior algebras.
 \begin{prop}[Prop.\ 10 of {\cite[Ch.III, \S7.7]{BourbakiI}}]\label{prop:ExtAlg}
 Let $V,W$ be vector spaces. Then there is a natural isomorphism of algebras
 \[\Lambda^*(V)\otimes\Lambda^*(W)\to \Lambda^*(V\oplus W),\]
 from the graded tensor product of the exterior algebras to the exterior algebra of the direct sum.
 \end{prop}

We will now restrict to the case, where the adjoint action of $G$ on $\mathfrak g$ is trivial. This includes, in particular, the case of abelian Lie groups.

An element of $\Omega_G^*(G\times (M/G))$ is an invariant section of 
\[S^*(\mathfrak g\dual)\otimes \Lambda^*\left(T\dual\!\left(G\times\faktor{M}{G}\right)\right)\to G\times \faktor{M}{G},\]
what by the splitting of the cotangent space and Proposition \ref{prop:ExtAlg} is a $G$-invariant section of 
\[S^*(\mathfrak g\dual)\otimes \Lambda^*\left(\pr_1^*T\dual G\right)\otimes \Lambda^*\left(\pr_2^*T\dual \faktor{M}{G}\right)\to G\times \faktor{M}{G}.\]
This is the same as a section of \[S^*(\mathfrak g\dual)\otimes \Lambda^*\left(\mathfrak g\dual\right)\otimes \Lambda^*\left(\pr_2^*T\dual \faktor{M}{G}\right)\to \faktor{M}{G},\] since the action of $G$ on $S^*(\g\dual)$ is trivial.
Pulling this section back to $M$ along the quotient map yields a $G$-invariant section of
\[S^*(\mathfrak g\dual)\otimes \Lambda^*\left(\mathfrak g\dual\right)\otimes q^*\Lambda^*\left(T\dual \faktor{M}{G}\right)\to M.\]
Composition with $\id\otimes \sharp\otimes \left(dq\left|_{\left(\mathfrak{g}^\sharp\right)^\bot}\right)\right.\invers$ turns this section to an $G$-invariant section of 
\[S^*(\mathfrak g\dual)\otimes \Lambda^*T\dual M\to M\]
and thus an element of $\Omega_G^*(M)$, because $X^\sharp_{gm}=g(g\invers Xg)^\sharp_{m}=g\cdot X^\sharp_{m}$. As any of these steps may be gone in the opposite direction, we have an isomorphism between $\Omega_G^*(M)$ and $\Omega_G^*(G\times (M/G))$.

Thus for free proper actions of abelian groups, there is no difference between $M$ and $G\times (M/G)$ in equivariant differential cohomology. The most easy example for a free proper action of a non-abelian Lie group on a manifold is the left multiplication of $S^3\subset \mathbb H$ on $S^7\subset \mathbb H$. We will leave this discussion to future research.
\\[2ex]
Let $E\to M$ be a $G$-equivariant vector bundle with free and proper $G$-action on the base and the total space. Given a connection on $\nabla$ on $E$, there is the question, whether this connection is a pullback from the quotient bundle
\[\begin{tikzcd}
    E\arrow{r}{\bar q}\arrow{d}&\overline{E}\defeq\faktor{E}{G}\arrow{d}\\M\arrow{r}{q}&\overline{M}\defeq\faktor{M}{G}.
\end{tikzcd}\]
Clearly, if the connection is a pullback, then every equivariant differential characteristic class $\hat c(\nabla)$ must lie in the image of 
\[\tilde q^*\from\hat H(\overline{M},\Z)\to \fatH_G(M,\Z),\]
where $\tilde q$ is the projection of the simplicial manifolds $G^\bullet\times M\to \{e\}^\bullet\times\bar M$. In particular, the connection must be $G$-invariant and the moment map must vanish (compare also \cite[Section 2.2]{Bunke}).

Now turn the question the other way around: Assume, the there is some collection of equivariant differential characteristic classes for a connection on $E\to M$, which all lie in the image of $\tilde q^*$. Does this imply that connection descends to the quotient bundle?

We want to remark to following observations according to an answer of this question: Let $\nabla$ be a connection on the equivariant complex vector bundle $E\to M$ of rank $n$. Then total equivariant Chern form is given by 
\[R(\hat c(\nabla)=\det\left(1+\frac{1}{2\pi i}R^\nabla+\mu^\nabla\right).\]
For any $X\in\mathfrak g$, this form induces a polynomial
\begin{align*}P_X(t)&=\det\left(1+\frac{1}{2\pi i}R^\nabla+\mu^\nabla(tX)\right)\\
	&=\det\left(1+\frac{1}{2\pi i}R^\nabla+t\mu^\nabla(X)\right)    
\end{align*}
in $t$. If the total equivariant Chern form lies in the image of the quotient map, then the degree of polynomial in $t$ is zero.

In the case of $R^\nabla=0$, $t^nP_X(\frac1t)$ is exactly the characteristic polynomial of $\mu^\nabla(X)$ and hence all eigenvalues of $\mu^\nabla(X)$ are zero, if total equivariant Chern form lies in the image of the quotient map. In general, this does not imply that $\mu^\nabla(X)$ is zero, but if there is a metric on $E$, we can say more.

Let $h$ be a hermitian metric on $E$ and $\nabla$ be compatible with $h$. Then $E$ is in correspondence to a principal $U(n)$-bundle and, as the Lie algebra $\mathfrak u(n)$ consists of anti-hermitian matrices, the image of $\mu^\nabla(X)$ at any point of $M$ is anti-hermitian. The Jordan normal form of an anti-hermitian matrix is diagonal, because the conjugate of an anti-hermitian matrix by an unitary one is anti-hermitian, \[(U^*AU)^*=U^*A^*U=-U^*AU\] and hence all $1$'s' in the first the upper diagonal must vanish.
Since an invariant connection descends, if and only if the moment map vanishes (compare \cite[Problem 2.24]{Bunke}), we have proven the following Proposition.
\begin{prop}
    Let $(E,h)\to M$ be a $G$-equivariant hermitian vector bundle, such that the $G$-action is free and proper, and let $\nabla$ be a $G$-invariant hermitian connection on $E$, such that the curvature $R^\nabla$ vanishes, then $\nabla$ descends to a connection on
\[\faktor{E}{G}\to \faktor{M}{G},\] if and only if the total Chern form vanishes.
\end{prop}

\subsection{Conjugation action on \texorpdfstring{$S^3$}{S3}}\label{sec:S3}
The manifold $S^3\subset\Real^4$ has a group structure. Recall that one defines on the vector space $\Real^4$ a real (non-commutative) division algebra, the quaternions, with three imaginary units $i,j,k$ squaring to $-1$ and satisfying $ij=-ji=k$. Now the space of unit quaternions is $S^3$ and has an induced multiplication. On the other hand, there is another description of the 3-sphere by the special unitary group of complex $2\times 2$-matrices
\[SU(2)=\left\{\begin{pmatrix} a&-\bar b\\b&\bar a\end{pmatrix}\middle| a,b\in \C, |a|^2+|b|^2=1\right\}.\]
The map
\[\begin{pmatrix} a&-\bar b\\b&\bar a\end{pmatrix}\mapsto a+j b\in S^3\subset\mathbb{H}\]
defines a group isomorphism between the two descriptions.

We want to investigate the conjugation action of $S^3$ on itself. Therefore note the following well-known fact (for a proof see e.g. \cite[Lemma 4.44]{Diss}):
\begin{lemma}
Half the trace or the real part of the quaternion is an invariant surjective mapping
\[\frac12\tr\from S^3\to [-1,1],\]
which induces an isomorphism of the quotient $\faktor{S^3}{SU(2)}\to[-1,1]$. The isotropy group of any point besides $1$ and $-1$ is isomorphic to $S^1$.
\end{lemma}

Another helpful picture of $S^3$ is obtained from stereographic projection with projection point $-1$. In formulas this is expressed as
\[\mathbb{H}\supset S^3\ni x=x_0+ix_1+jx_2+kx_3\mapsto \frac1{1+x_0} (x_1,x_2,x_3)\in \Real^3\cup \{\infty\},\]
where $1\in\mathbb{H}$ is mapped to the $0\in \Real^3$ and $-1$ to $\infty$. Taking subsets $S^3\subset\mathbb{H}$ of fixed real value $x_0$, then these are mapped to 2-sphere of Radius $\sqrt{\frac{1-x_0}{1+x_0}}$.
The conjugation action acts transitive on each of these 2-spheres and leaves the midpoint and $\infty$ fixed.
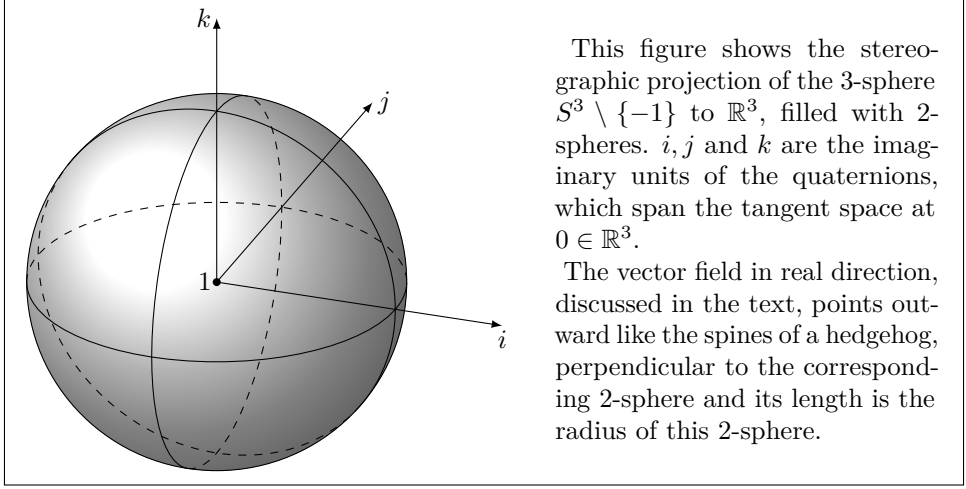
\begin{center}\setlength{\fboxsep}{4pt}\framebox{\hspace*{-4pt}
\begin{minipage}[c]{7.5cm}
\begin{tikzpicture} 

\def\R{2.5} \def\angEl{25} \def\angAz{-20} 

\pgfmathsetmacro\H{\R*cos(\angEl)} \tikzset{xyplane/.estyle={cm={cos(\angAz),sin(\angAz)*sin(\angEl),-sin(\angAz),
                              cos(\angAz)*sin(\angEl),(0,0)}}}
\LongitudePlane[xzplane]{\angEl}{\angAz}
\LatitudePlane[equator]{\angEl}{0}

\fill[ball color=white] (0,0) circle (\R); \draw (0,0) circle (\R);

\coordinate[mark coordinate] (O) at (0,0);
\draw (0,0) node[left] {$1$};

\DrawLatitudeCircle[\R]{0} \DrawLongitudeCircle[\R]{\angAz} \DrawLongitudeCircle[\R]{\angAz+90} 

\draw[xyplane,<->] (1.6*\R,0) node[below] {$i$} -- (0,0) -- (0,2.4*\R)
    node[right] {$j$};
\draw[->] (0,0) -- (0,1.4*\R) node[left] {$k$};
\end{tikzpicture}
\end{minipage}\hspace{-14pt}
\begin{minipage}[c]{5cm}
\hspace*{\parindent} This figure shows the stereographic projection of the 3-sphere $S^3\setminus\{-1\}$ to $\Real^3$, filled with 2-spheres. $i,j$ and $k$ are the imaginary units of the quaternions, which span the tangent space at $0\in\Real^3$. 
\par
\hspace*{\parindent} The vector field in real direction, discussed in the text, points outward like the spines of a hedgehog, perpendicular to the corresponding 2-sphere and its length is the radius of this 2-sphere.
\end{minipage}\hspace*{3pt}
}
\end{center}
Let $f\in C^\infty(S^3)^{S^3}$. It is clear that the map does only depend on the real value or, in the other picture, not on the point itself, but only on the 2-sphere, on which the point is located. To be smooth, the function must depend smoothly on the real value and the different direction must fit at $1$ and $-1$. As the function has the same value in any direction of $1$, fitting smoothly means that all odd derivatives must vanish. Thus 
\begin{align*}
C^\infty(S^3)^{S^3}&\cong \left\{f\in C^\infty([-1,1])\middle|\frac{d^kf}{dt^k}(-1)=\frac{d^kf}{dt^k}(1)=0, \text{ for all odd } k>0\right\}\\&\subset C^\infty([-1,1],\C)
\end{align*}
Now, we are going to understand invariant differential forms of the conjugation action on $S^3$.

Let $\omega\in \Omega^1(S^3)^{S^3}$. Let $v$ be a tangent vector on one of the two fixed points. Then there exists $g\in S^3$, s.t. $g\invers v g=-v$, hence an invariant one form must be zero on the fixed points. As the real part of the quaternion is invariant under conjugation, the vector field pointing in this direction, projects to a invariant tangent field on $S^3$, which vanishes only at $1$ and $-1$. In the $\Real^3$ picture, this is the radial vector field pointing outward everywhere. Let $X$ now denote the normalization of this vector field on $S^3\setminus\{1,-1\}$, and $\omega_0$ the one form dual to $X$. Let $\omega_1=\omega-(\iota(X)\omega)\omega_0$, where $\iota$ is the contraction of the form by the field. A priori this forms are only defined on $S^3\setminus\{1,-1\}$, but as $\omega$ is zero at $1$ and $-1$, we can extend $(\iota(X)\omega)\omega_0$ and $\omega_1$ by zero to obtain a smooth form on all of $S^3$. Taking any slice of $S^3$ with fixed real part in $(-1,1)$, this is isomorphic to $S^2$ and $\omega_1$ actually is a one form on each of these 2-spheres. The $S^1$-isotropy found above, acts non-trivially on tangent vectors. Hence with the same argument as above (Rotating the tangent vector to minus itself) one sees, that $\omega_1$ actually is zero. Thus $\omega=(\iota(X)\omega)\omega_0$. Let $f$ be the integral of $\iota(X)\omega\in C^\infty(S^3)^{S^3}\subset C^\infty([-1,1])$ over the interval, then $\omega=df$ and $f'(1)=f'(-1)=0$ as $\omega$ vanishes at the fixed points. 
Thus we have shown 
\[\faktor{\Omega^1(S^3)^{S^3}}{d C^\infty(S^3)^{S^3}}=0.\]
Let $\omega\in \Omega^2(S^3)^{S^3}$. Contracting with the radial field $X$ as defined in the last paragraph yields $\iota(X)\omega=f\omega_0$, for some function $f$. As $\iota^2=0$, $f=0$. Thus, restricting $\omega$ to each of the level of fixed real part in the open interval, one obtains a multiple of the volume form on $S^2$. At the fixed points one gets a $SO(3)$-invariant 2-form on $\Real^3\!$, since the adjoint action on the Lie algebra of $SU(2)$ is how one defines the double cover of $SU(2)\to SO(3)$. But there is no non-zero skew-symmetric matrix commuting with the whole $SO(3)$. Thus $\omega$ must vanish on the fixed points. Moreover, as any invariant 1-form is exact,
\begin{align*}\faktor{\Omega^2(S^3)^{S^3}\hspace{-1ex}}{d\Omega^1(S^3)^{S^3}}&=\Omega^2(S^3)^{S^3}\\&\cong \left\{f\in C^\infty([-1,1])\middle| f(-1)=f(1)=0,\frac{d^kf}{dt^k}(\pm1)=0, k \text{ odd}\right\}\!.
\end{align*}
 A volume form on the manifold induces an isomorphism $\Omega^3(S^3)\cong C^\infty(S^3)$. Since the standard volume is invariant,
 we get an isomorphism for invariant forms and functions. Let $X\in\mathfrak{s}^3\subset\mathbb{H}$. Then 
\[X^\sharp(m)=\ddt (1+tX)m(1-tX)=Xm-mX.\]
Thus for $\omega\in \Omega^3(S^3)^{S^3}$ 
\begin{multline}\iota(X^\sharp)\omega(m)=\iota(Xm-mX)\omega(m)\overset{\omega=\Ad^*\omega}{=}\iota(Xm)\omega(m)-\iota(mX)Ad_m^*\omega(m)\\=\iota(Xm)\omega(m)-\iota(m\invers mXm )\omega(m)=0. 
\end{multline}
 Moreover, $d$ vanishes on top forms, hence the Cartan differential on $\Omega^3(S^3)^{S^3}$ is zero. As $S^3$ has empty boundary
\[\int_{S^3} \from d\Omega^2(S^3)^{S^3}\to \C\] is the zero map by Stokes theorem. Thus
\[\faktor{\Omega^3(S^3)^{S^3}}{d\Omega^2(S^3)^{S^3}}\to \C, \omega\mapsto \int_{S^3}\omega.\]
is a well defined injective homomorphism. From the calculation of the cohomology below, we see, that it is surjective.

What is the classical equivariant cohomology of the conjugation action of $S^3$ with values in $R\in\{\Z,\C,\C/\Z\}$? Taking the simplicial manifold model for $ES^3\times_{S^3} S^3$ and a cellular resolution with cell structure on $S^3$ given by one zero cell, corresponding to the neutral element of $S^3$, and one three cell, then all simplicial maps are cellular and we obtain the following double complex with the cellular resolution horizontally to the right and the simplicial complex in vertical direction downwards (compare page \pageref{sec:simpcellcoho}).
\[
\begin{tikzcd}
S^3\hspace{2em}&R\arrow{r}\arrow{d}{0} &0\arrow{r}\arrow{d} &0\arrow{r}\arrow{d} &R\arrow{r}\arrow{d}{\del^{(0)}} &0\arrow{r}\arrow{d} &\dots\\
S^3\times S^3\hspace{2em}&R\arrow{r}\arrow{d}{1} &0\arrow{r}\arrow{d} &0\arrow{r}\arrow{d} &R^2\arrow{r}\arrow{d}{\del^{(1)}} &0\arrow{r}\arrow{d} &\dots\\
(S^3)^2\times S^3\hspace{2em}&R\arrow{r}\arrow{d}{0} &0\arrow{r}\arrow{d} &0\arrow{r}\arrow{d} &R^3\arrow{r}\arrow{d}{\del^{(2)}} &0\arrow{r}\arrow{d} &\dots\\
(S^3)^3\times S^3\hspace{2em}&R\arrow{r}\arrow{d}{1} &0\arrow{r}\arrow{d} &0\arrow{r}\arrow{d} &R^4\arrow{r}\arrow{d}{\del^{(3)}} &0\arrow{r}\arrow{d} &\dots\\
(S^3)^4\times S^3\hspace{2em}&R\arrow{r}\arrow{d}{0} &0\arrow{r}\arrow{d} &0\arrow{r}\arrow{d} &R^5\arrow{r}\arrow{d} &0\arrow{r}\arrow{d} &\dots\\
\vdots&\vdots&\vdots&\vdots&\vdots&\vdots&
\end{tikzcd}
\]
The $R$ in the 0-column corresponds to the zero cell and the $R^k$ in the 3-column corresponds to the $k$ 3-cells in $(S^3)^{\times k}$. The 3-cells in $S^3\times S^3$ are $S^3\times \{e\}$ and $\{e\}\times S^3$ and in $S^3\times S^3\times S^3$ are $S^3\times \{e\}\times \{e\}$,$\{e\}\times S^3\times \{e\}$ and $\{e\}\times \{e\}\times S^3$. One calculates directly for the conjugation action, that $\del^{(0)}=0$ and $\del^{(1)} (a,b)=(0,0,b)$, where the $i$-th entry corresponds to the $i$-th cell. Hence we obtain
\[H^k_{S^3}(S^3,R)=\begin{cases}
                  R & k=0,3,4\\
		  0 & k=1,2\\
                 \end{cases}
\]
and can interpret this geometrically: The third cohomology is generated by the 3-cell in $S^3$ and the fourth cohomology is generating by the `acting' 3-cell $S^3\times \{e\}\subset S^3\times S^3$.

Now the next proposition follows, in the main, by applying the hexagons \eqref{LieHexagon2} and \eqref{eq:CartanHexagon}.
\begin{prop}
For the conjugation action of the 3-sphere $S^3=SU(2)$ on itself, we have
\[\hat H^n_{S^3}(S^3,\Z)=\begin{cases}
                        \Z&n=0\\
                        C^\infty(S^3)^{S^3}/\Z&n=1\\
                        0&n=2\\
                        \Omega^2(S^3)^{S^3}\oplus \Z d\text{vol}_{S^3}\subset \Omega^3(S^3)^{S^3}&n=3\\ 
                        \C/\Z\oplus\Z&n=4\\ 
                        H^n_{S^3}(S^3,\Z)&n\geq5
                       \end{cases}
\]
and
\[\fatH^n_{S^3}(S^3,\Z)=\begin{cases}
                        \Z&n=0\\
                        C^\infty(S^3)^{S^3}/\Z&n=1\\
                        0&n=2\\
                        \Omega^2(S^3)^{S^3}\oplus \Z d\text{vol}_{S^3}\subset \Omega^3(S^3)^{S^3}&n=3\\ 
                        \C/\Z\oplus\Z\oplus\Omega^1(S^3)^{S^1}\!\!/C^\infty(S^3)^{S^1}&n=4.
                                               \end{cases}
\]
\end{prop}
\begin{proof}
For $\hat H^n_{S^3}(S^3,\Z)$, the only open question is the case $n=4$. There one obtains a short exact sequence $0\to \C/\Z\to \hat H^4_{S^3}(S^3,\Z)\to \Z\to 0$ from the hexagon. This sequence splits, because $\C/\Z$ is an injective abelian group.

In the case of $\fatH^4_{S^3}(S^3,\Z)$ one has the following hexagon from \eqref{eq:CartanHexagon}:
\begin{equation}\label{eq:H4S3erstes}
\begin{scaledcd}{.7}
\phantom{.} &\hspace{-2em}\C\oplus \faktor{(\mathfrak{s}^3)\dual\otimes\Omega^1(S^3)))^{S^3}}{d((\mathfrak{s}^3)\dual\otimes\Omega^0(S^3)))^{S^3})} \arrow{dr}{a} \arrow{rr}{0\oplus (d+\iota)}& & ((\mathfrak{s}^3)\dual\otimes\Omega^2(S^3)))^{S^3}\oplus S^2((\mathfrak{s}^3)\dual)\otimes\Omega^0(S^3)))^{S^3})_\cl\arrow{dr}{}\hspace{-3em}&\\
\C\arrow{ru}{}\arrow{dr}{}&&\fatH_{S^3}^4(S^3,\Z) \arrow{ru}{R}\arrow[two heads]{dr}{I}&&\C\\
&\C/\Z \arrow[hook]{ur}{} \arrow{rr}{0}&&\Z \arrow{ur}{}&
\end{scaledcd}
\end{equation}

As discussed above $\mathfrak{s}^3=\Real i+\Real j+\Real k \subset\mathbb{H}$ and $S^3$ acts transitive on the unit sphere of this space. Moreover, the subgroup of $S^3$, which leaves $i\in\mathfrak{s}^3$ invariant, is exactly $S^1\subset \C\subset\mathbb{H}$. Hence 
\begin{align*}\left(\left(\mathfrak{s}^3\right)\dual\otimes\Omega^k\left(S^3\right)\right)^{S^3}&\cong \Omega^k\left(S^3\right)^{S^1}\\
    \left(\omega\from \mathfrak{s}^3\to \Omega^k\left(S^3\right)\right)&\mapsto \omega(i)
\end{align*}
and, since the first and second de Rham cohomology of $S^3$ vanish, averaging over the $S^1$ implies that $d\from \Omega^1(S^3)^{S^1}/dC^\infty(S^3)^{S^1} \to \Omega_\cl^2(S^3)^{S^1}$ is an isomorphism.

Further, let 
\[(\omega,f)\in\left(\left((\mathfrak{s}^3)\dual\otimes\Omega^2(S^3)\right)^{S^3}\oplus \left(S^2((\mathfrak{s}^3)\dual)\otimes\Omega^0(S^3)\right)^{S^3}\right)_\cl,\]
i.e., $d\omega=0$ and $df=-\iota\omega$. Then $\omega=d\eta$ for one and only one 
\[\eta\in \faktor{(\mathfrak{s}^3)\dual\otimes\Omega^1(S^3)))^{S^3}\!\!}{d\left(\left(\left(\mathfrak{s}^3\right)\dual\otimes\Omega^0\left(S^3\right)\right)^{S^3}\right)}\] 
and $df=-\iota d\eta$. On the other $f$ is given by a symmetric $3\times 3$ matrix of smooth functions on $S^3$
\[\begin{pmatrix}
   f_{ii}&f_{ij}&f_{ik}\\f_{ji}&f_{jj}&f_{jk}\\f_{ki}&f_{kj}&f_{kk}\\
  \end{pmatrix}
\] and this matrix is determined, up to a constant matrix denoted by $A$, by the form $\eta$.
By the transitive action of $S^3$ on the Lie algebra, it is clear, that the information of the matrix is contained in $f_{ii}$ and $f_{ij}$. The conjugation by the element $\frac{1+k}{\sqrt 2}\in S^3$ translates the pair $(i,j)$ to $-(j,i)$. Hence $f_{ij}=-Ad^*_{\frac{1+k}{\sqrt 2}} f_{ij}$. Thus the off-diagonal terms of the symmetric matrix $A$ must vanish and hence $A$ must be a multiple of identity matrix. 

Thus, we have described an isomorphism 
\begin{align*}\C\oplus \faktor{\Omega^1(S^3)^{S^1}\!\!\!\!}{C^\infty(S^3)^{S^1}}&\to \left(\left((\mathfrak{s}^3)\dual\otimes\Omega^2(S^3)\right)^{S^3}\oplus \left(S^2((\mathfrak{s}^3)\dual)\otimes\Omega^0(S^3)\right)^{S^3}\right)_\cl\\
(A,\eta)&\mapsto (f,\omega).
\end{align*}
Applying this isomorphism, the hexagon \eqref{eq:H4S3erstes} changes to
\[
\begin{scaledcd}{1}
\faktor{\Omega^1(S^3)^{S^1}\!\!\!\!}{C^\infty(S^3)^{S^1}} \arrow[hook]{dr}{a} \arrow[hook]{rr}{\id}& & \faktor{\Omega^1(S^3)^{S^1}\!\!\!\!}{C^\infty(S^3)^{S^1}}\arrow{dr}\oplus \C&\\
&\fatH_{S^3}^4(S^3,\Z)\arrow{ru}{R}\arrow[two heads]{dr}{I}&&\C,\\
\C/\Z \arrow[hook]{ur}{} \arrow{rr}{0}&&\Z \arrow{ur}{}&
\end{scaledcd}
\]
where again the top line, the bottom line and the diagonals are exact. The map $a$ is injective because the inclusion in the top line factors as $R\circ a$.
\end{proof}
\subsection{Actions of finite cyclic groups on the circle}
Let $\Zp=\Z/p\Z$ denote the cyclic group with $p$ elements. There is an action of $C_p$ on any odd sphere $S^{2n-1}\subset \C^n$, where a fixed generator acts by multiplication with $e^{\frac{1}{p}2\pi i}$. This diagonal action is also unitary on the infinite-dimensional separable Hilbert space $l^2(\Natural,\C)$ and hence induces an action on the unit sphere $S^\infty$. The inclusions of $\C^n$s as first coefficients induce equivariant inclusions
\[S^1\to S^3\to \dots \to S^\infty\]
The sum of the tangent bundle and the normal bundle of $S^1\subset\C$ is a complex line bundle, $TS^1\oplus N\cong S^1\times \C$, which we equip with the connection $\nabla$, whose associated parallel transport respects the decomposition in tangent and  normal space. Hence, the holonomy once around the circle equals $2\pi$, thus is trivial. The sphere bundle (with respect to the standard metric) of $TS^1\oplus N$ is the trivial $S^1$ bundle on $S^1$ with the $S^1$-invariant connection. Now we have a pullback diagram of bundles with connection with equivariant maps
\[\begin{tikzcd}
    \left(S^1\times S^1,\nabla\right)\arrow{d}\arrow{r}&H=S^3\times_{S^1}S^1\arrow{d}\arrow{r}&S^\infty\arrow{d}\\
S^1\arrow{r}{f}&S^3\arrow{r}&S^\infty/S^1.
\end{tikzcd}\]
Moreover the first Chern class $c_1(S\infty\to S^\infty/S^1)\in H^2(S^\infty/S^1)=H^2(BS^1)$ is a generator.  Now for $\hat H^2_{C_p}(S^3,\Z)$ we have the diagram
\[\begin{tikzcd}\phantom. &\hat H^2_{C_p}(S^3,\Z)\arrow{dr}{I}&\\H^1_{C_p}(S^3,\C/\Z)\arrow[hook]{ur}\arrow{rr}{-\beta}&&H^2_{C_p}(S^3,\Z).\end{tikzcd}\] As first and second cohomology are torsion, the Bockstein is an isomorphism, given by multiplication with $p$. As the connection on $H$ is flat, $\hat c_1(H)$ actually is a class in $H^1_{C_p}(S^3,\C/\Z)$.
Let the cycle $	au=\left[0,\frac1p\right]\subset \Real/\Z\cong S^1$ be a fundamental domain of the $C_p$ action on $S^1$. Evaluation at $f(	au)$ induces the isomorphism $H^1_\Zp(S^3,\C/\Z)\to \left(\frac1p\Z\right)/\Z$ under which $c_1(H)$ is mapped to $\frac1p$. Pulling back the class along $f$ shows
\[\hat c_1(TS^1\oplus N)=\frac1p\in \C/\Z.\]
A finer analysis shows that the bundle $S^1\times S^1\to S^1$, where $C_p$ acts by multiplication with $e^{\frac{q}{p}2\pi i}$ on the fiber and $e^{\frac{1}{p}2\pi i}$ on the base space, has first equivariant differential Chern class $\frac{q}p\in \C/\Z$. One may interpret this as a measurement of holonomy along the fundamental domain.

\subsection{\texorpdfstring{$G$}{G}-Representations}
In this section, we want to investigate actions of Lie groups on $\Real^n$. This will lead to some implication to equivariant immersions. Equivariant immersions will be subject of further investigation. To generalize the well-known methods of characteristic classes applied to immersion, one has, in particular, to define multiplicative structures in equivariant differential cohomology and generalize the Whitney-Sum-Formula.
\\[2ex]
An orthogonal representation of the Lie group $G$ on $\Real^n$ (with the standard metric) is given by a map $
ho\from G\to O(n)$. This induces an action on the tangent bundle $(T\Real^n,\nabla)=(\Real^n\times\Real^n,d)$ with the trivial connection. As $d^2$ is zero, the curvature vanishes, i.e. $R^\nabla=0$. But since the trivialization of the tangent bundle is not an equivariant trivialization the moment map will not vanish in general.
\begin{align*}
\mu^\nabla(X)\varphi(m)&=d\varphi_m(X^\sharp)+\ddt(
ho(\exp(tX))\varphi(
ho(\exp(-tX))m)\\
&=d\varphi_m(X^\sharp)+d
ho(X)\varphi(m)-d\varphi_m(X^\sharp)\\    
&=d
ho(X)\varphi(m)   
\end{align*}

Hence for any equivariant differential characteristic class $\hat c$ with corresponding invariant polynomial $P\in I^*(O(n))$, one has 
\[R(\hat c(G\curvearrowright \Real^n))=P(\mu^\nabla+R^\nabla)=P(d
ho)\in S^*(\g\dual)\otimes\Omega^0_\cl(\Real^n)=S^*(\g\dual).\]

In particular, the characteristic form (and hence the class) will not vanish in general for this flat bundle.

\subsection{Towards obstruction to immersions?} 
A major application of characteristic classes in the non-equivariant case is given by obstructions to immersions -- more precisely, the characteristic classes give lower bounds to the minimal codimension of an immersion. In the world of classical characteristic classes this can be found, e.g., in \cite[Theorem 4.8]{MilnorStasheff}). Differential characteristic classes apply for a result, that conformal immersions have a stronger bound for the minimal codimension, than smooth immersions (see \cite{Millson} and \cite[\S 6]{CS} for the original work and \cite{Moore} for a partly strengthened version).

The arguments therefore go as follows:  Let $M$ be an (Riemannian) manifold and $f\from M\to \Real^n$ an (isometric) immersion. Then there is a normal bundle $NM\to M$, such that
\[TM\oplus NM=f^*T\Real^n.\]
Since the Chern classes on the right hand side vanish, the the total Chern class of $NM$ must be the inverse (with respect to the cup product) of $TM$. This implies restrictions to the values of these classes. 
Moreover, in the Riemannian case, the Levi-Civita connection on $M$ is compatible with the pullback connection $\nabla_f$ of the trivial connection on $\Real^n$ to $TM\oplus NM$. This implies similar statements for the differentially refined characteristic classes of the Riemannian connections.

More explicitly John Millson calculates the first differential Pontryagin class of some lens spaces and shows, that these do not immerse conformally into $\Real^n$ with certain codimension, where smooth immersions exist. It is, with our theory, straight forward to restate these examples for the lens space action of a finite cyclic group on the 3-sphere, which should be immersed into a trivial representation. It is subject of further research to study equivariant conformal immersions into non-trivial representations.

\section*{Acknowledgements}
The first author wants to thank the International Max Planck Research School {\it Mathematics in the Sciences} for financial support. This research was supported by ERC Starting Grant No.\ 277728.
\printbibliography
\end{document}